\documentclass[11pt]{amsart}
\usepackage[hypertex]{hyperref}
\usepackage{amsmath,amsthm,amssymb, amscd, amsfonts}
\usepackage[all]{xy}
\usepackage{leftidx}
\usepackage{bbold}
\usepackage[latin1]{inputenc}
\usepackage{enumerate}

\theoremstyle{plain}

\newtheorem{theorem}{Theorem}[section]
\newtheorem{corollary}[theorem]{Corollary}

\newtheorem{proposition}[theorem]{Proposition}
\newtheorem{lemma}[theorem]{Lemma}

\theoremstyle{definition}
\newtheorem{definition}[theorem]{Definition}

\theoremstyle{remark}
\newtheorem{remark}[theorem]{Remark}

\numberwithin{equation}{section}\theoremstyle{plain}

\newcommand{\uno}{ \mathbb{1}}
\newcommand{\otb}{{\overline{\otimes}}}

\newcommand{\vphi}{\varphi}
\renewcommand{\1}{\textbf{1}}

\newcommand{\lt}{\prec}
\newcommand{\rt}{\succ}

\newcommand{\C}{{\mathcal C}}
\newcommand{\D}{{\mathcal D}}

\newcommand{\Z}{{\mathcal Z}}
\newcommand{\Zz}{{\mathbb Z}}
\newcommand{\M}{\mathcal{M}}

\newcommand{\toto}{\longrightarrow}

\newcommand{\E}{{\mathcal E}}

\newcommand{\Rep}{\operatorname{Rep}}
\newcommand{\Supp}{\operatorname{Supp}}

\newcommand{\KER}{\mathfrak{Ker}}

\newcommand\corep{\operatorname{comod}\!\mbox{-}}

\newcommand\Aut{\operatorname{Aut}}
\newcommand\Irr{\operatorname{Irr}}
\newcommand\FPdim{\operatorname{FPdim}}

\newcommand\vect{\operatorname{Vec}}

\newcommand\id{\operatorname{id}}

\newcommand\End{\operatorname{End}}

\newcommand\op{\operatorname{op}}
\newcommand\Hom{\operatorname{Hom}}

\begin{document}
\title[Crossed actions of matched pairs]{Crossed actions of matched pairs of
groups on tensor categories}
\author{Sonia Natale}
\address{Facultad de Matem\'atica, Astronom\'\i a y  F\'\i sica.
Universidad Nacional de C\'ordoba. CIEM -- CONICET. Ciudad
Universitaria. (5000) C\'ordoba, Argentina}
\email{natale@famaf.unc.edu.ar
\newline \indent \emph{URL:}\/ http://www.famaf.unc.edu.ar/$\sim$natale}

\thanks{Partially supported by  CONICET, SeCYT--UNC and Alexander von Humboldt
Foundation}



\date{\today}

\begin{abstract} We introduce the notion of $(G, \Gamma)$-crossed action on a
tensor category, where $(G, \Gamma)$ is a matched pair of finite groups. A
tensor category is called a $(G, \Gamma)$-crossed tensor category if it is
endowed with a  $(G, \Gamma)$-crossed action.  We show that every  $(G,
\Gamma)$-crossed tensor category $\C$ gives rise to a tensor
category $\C^{(G, \Gamma)}$ that fits into an exact sequence of tensor
categories $\Rep G \toto \C^{(G, \Gamma)} \toto \C$. We also define the notion
of a $(G, \Gamma)$-braiding in a $(G, \Gamma)$-crossed tensor category, which
is connected with certain set-theoretical solutions of the QYBE. This extends
the notion
of $G$-crossed braided tensor category due to Turaev.  We show that if $\C$ is a
$(G, \Gamma)$-crossed tensor category equipped with a $(G, \Gamma)$-braiding,
then the tensor category $\C^{(G, \Gamma)}$ is a braided tensor category in a
canonical way.
\end{abstract}

\maketitle

\section{Introduction}

Besides from their inherent algebraic appeal, monoidal and tensor categories are
relevant structures in many areas of mathematics and mathematical physics.  The
endeavour around the far-reaching problem of their classification has seen a
considerable outgrowth in the last decades. Widespread examples of tensor
categories are provided by Hopf algebras and its generalizations by means of its
representation theory.

The main goal of this paper is to present a construction of a class of tensor
categories that generalizes and puts into a unified perspective certain renowned
classes of examples.  

The input for this construction consists of a matched pair of finite groups $(G.
\Gamma)$ plus a tensor category $\C$ endowed with a $\Gamma$-grading and an
action of $G$ by  autoequivalences (which are not necessarily tensor functors):
$$\C = \bigoplus_{s \in \Gamma} \C_s, \qquad \rho: \underline{G}^{\op} \to
\underline{\Aut}(\C),$$
that are related to each other in an appropriate sense. For reasons that might
well become apparent in the sequel, we call such a data a \emph{$(G,
\Gamma)$-crossed action} on $\C$. We say that $\C$ is a \emph{$(G,
\Gamma)$-crossed tensor category}, if it is endowed with a $(G, \Gamma)$-crossed
action. See Definition \ref{crossed-action}.

\medbreak Recall that a \emph{matched pair of groups} is a collection $(G,
\Gamma)$, where $G$ and $\Gamma$
are groups  endowed with  mutual  actions by permutations
$$\Gamma \overset{\vartriangleleft}\longleftarrow \Gamma \times G
\overset{\vartriangleright}\longrightarrow G$$ satisfying the following
conditions:
\begin{equation}\label{matched}
s \vartriangleright gh  = (s \vartriangleright g) ((s
\vartriangleleft g) \vartriangleright h), \quad st
\vartriangleleft g  = (s \vartriangleleft (t \vartriangleright g))
(t \vartriangleleft g), \end{equation} for all $s, t \in \Gamma$,
$g, h \in G$. 

\medbreak The requirements in our definition of a $(G, \Gamma)$-crossed tensor
category are that, for all $g \in G$, $s\in \Gamma$, 
$$\rho^g(\C_s) = \C_{s \lhd g},$$  
and the existence of natural isomorphisms $$\gamma^g_{X, Y}: \rho^g(X \otimes Y)
\to
\rho^{s\rhd g}(X) \otimes \rho^g(Y), \qquad  X \in \C, \,  Y \in
\C_s,$$ subject to certain rather natural compatibility conditions.

\medbreak From a $(G, \Gamma)$-crossed tensor category $\C$ we produce a new
tensor category that we denote $\C^{(G, \Gamma)}$. The tensor product in
$\C^{(G, \Gamma)}$ is built from the tensor product of $\C$ and the natural isomorphisms $\gamma$. This is done in Theorem
\ref{tens-prod}. 

\medbreak The main tool in the proof of Theorem \ref{tens-prod} is the notion of a
\emph{Hopf monad}, introduced in \cite{bv}, \cite{blv}. This notion and some of
its main features are recalled in Subsection \ref{hmnds}. It turns out that the
data underlying a $(G, \Gamma)$-crossed tensor category $\C$ give rise to a
monad $T$ on $\C$ in such a way that the category $\C^T$ of $T$-modules in $\C$
identifies with $\C^{(G, \Gamma)}$.  We show that, with respect to a suitable
comonoidal structure arising from the $(G, \Gamma)$-crossed action on $\C$, $T$
is in fact a Hopf monad,  which allows to conclude that $\C^{(G, \Gamma)}$ is a
tensor category.

We have that $\C^{(G, \Gamma)}$ is a
finite tensor category if and only if the neutral homogeneous component $\D =
\C_e$ of the associated $\Gamma$-grading is a finite tensor category. On the
other side, 
$\C^{(G, \Gamma)}$ is a fusion category if and only if  $\D$ is a fusion
category and the characteristic of $k$ does not divide the order of $G$
(Proposition \ref{finite-fusion}).

\medbreak We show  that, like in the case of an equivariantization under a group
action by tensor autoequivalences, the category $\C^{(G, \Gamma)}$ fits into an
exact sequence $$\Rep G \toto \C^{(G, \Gamma)} \toto \C,$$ in the sense of the
definition given in \cite{tensor-exact}. See Theorem \ref{exact-sequence}.
However, this is not an equivariantization exact sequence, unless the action 
$\rhd: \Gamma \times G \longrightarrow G$ is (essentially) trivial. Dually, the
category $\C^{(G, \Gamma)}$ is not a $\Gamma$-graded tensor category, unless the
action $\lhd: \Gamma \times G \longrightarrow \Gamma$ is (essentially) trivial.
See Propositions \ref{equiv-triv} and \ref{gr-triv}.

\medbreak Let $G$ be a group. Motivated by his developements in Homotopy
 Quantum Field Theory, Turaev introduced the notion of $G$-crossed
braided categories \cite{turaev}, which serve as a tool in the construction of
invariants of 3-dimensional $G$-manifolds. 
M\" uger showed in \cite{mueger-crossed} (see also \cite{kirillov}) that 
$G$-crossed braided categories arise
from the so-called Galois extensions of braided tensor categories.  

As it turns out, the $G$-crossed categories underlying $G$-crossed braided
categories of Turaev yield examples of crossed actions of a matched pair.
Indeed, the right adjoint action $\rhd: G \times G \longrightarrow G$ and the
trivial action $\lhd: G \times G \longrightarrow G$ make $(G, G)$ into a 
matched pair of groups. The conditions in Definition \ref{crossed-action} of a
$(G, G)$-crossed action on a tensor category $\C$ boil down in this case to the
conditions defining a $G$-crossed tensor category $\C$.

\medbreak Let $\C$ be a $(G, \Gamma)$-crossed tensor category. We define in this
paper a \emph{$(G, \Gamma)$-braiding} in $\C$ as a triple $(c, \vphi, \psi)$,
where  $\vphi, \psi:\Gamma \to G$ are group homomorphisms and  $c$ is a
collection of natural isomorphisms 
$$c_{X, Y}: X \otimes Y \to \rho^{t^{-1} \rhd
\vphi(s^{-1})}(Y) \otimes \rho^{\psi(t)}(X), \qquad X \in \C_s, \, Y \in \C_t,$$
satisfying certain compatibility conditions. See Definition
\ref{crossed-braiding}. 

\medbreak Recall that a set-theoretical solution of the Quantum Yang-Baxter
Equation is an invertible map $r: X \times X \to X \times X$, where $X$ is a
set, satisfying the condition $r^{12}r^{13}r^{23} = r^{23}r^{13}r^{12}$, as maps
$X \times X \times X \to X \times X \times X$. A theory of set-theoretical
solutions of the QYBE was developed in \cite{ess}, \cite{lyz2}, \cite{s}.
Our definition of a $(G, \Gamma)$-braiding is related to the set-theoretical
solutions of the  QYBE
equation on the set $\Gamma$ studied in \cite{lyz2},
corresponding to appropriate actions of the group $\Gamma$ on
itself. We discuss this relation in Subsection \ref{set-YBE}.

\medbreak We show that  a $(G, \Gamma)$-braiding in $\C$ gives
rise to a braiding in $\C^{(G, \Gamma)}$, thus providing examples of braided
tensor categories.   See Theorem \ref{crossed-braided}.

In the case where $\C$ is a $G$-graded tensor category, regarded as before as
$(G, G)$-crossed tensor category, a $G$-braiding $c$ in $\C$ is the same thing
as a $(G, G)$-braiding $(c, \vphi, \psi)$, where $\psi = \id_G: G \to G$ is the
identity group homomorphism and $\vphi$ is the trivial group homomorphism
(Proposition \ref{data-gb}).  

\medbreak Matched pairs of groups are the main ingredients in the origin of one
of the first classes of examples of non-commutative and non-cocommutative Hopf
algebras discovered by G. I. Kac in the late 60's \cite{kac} (see also
\cite{majid}, \cite{masuoka}, \cite{takeuchi}). These Hopf algebras are most
commonly called \emph{abelian bicrossed products} or \emph{abelian extensions};
they are characterized by the attribute of fitting into an exact sequence of
Hopf algebras
\begin{equation}\label{h-abext}k \toto k^\Gamma \toto H \toto k G \toto
k,\end{equation} 
where $G$ and $\Gamma$ are finite groups which, \textit{a fortiori}, form a
matched pair $(G, \Gamma)$. This class of Hopf algebras, as well as its
generalizations in different contexts, has been intensively studied in the
literature. 

\medbreak We show that the representation category of an abelian extension of
finite dimensional Hopf algebras fits into our construction. More precisely, we
use the cohomological data determining an abelian exact sequence as in
\eqref{h-abext} to provide the tensor category $\C(\Gamma)$ of finite
dimensional $\Gamma$-graded vector spaces with a $(G, \Gamma)$-crossed action,
such that the outcoming tensor category $\C^{(G, \Gamma)}$ is strictly
equivalent to the tensor category of finite dimensional representations of $H$.

\medbreak 
Along this paper $k$ will be an algebraically closed field. Our discussion
focuses on the framework of tensor categories over $k$. Several pertinent
definitions and facts about tensor categories are recalled in Subsection
\ref{tc}.  We refer the reader to \cite{BK}, \cite{egno}, for a detailed
treatment of the subject.

We stress that most notions and results in this paper  (c.f. Sections
\ref{main-const}, \ref{the-cat} and \ref{braidings}) can be formulated as well
for more general (not necessarily finite) matched pairs of  groups $(G, \Gamma)$
and
monoidal categories $\C$, under milder assumptions. The finiteness restriction
on
the pair $(G, \Gamma)$ are imposed by the intrinsic finiteness of tensor
categories.

\medbreak The contents of the paper are organized as follows. In Section
\ref{prels} we overview the distinct concepts and basic facts on the main
structures entering into the picture: matched pairs of groups,  tensor
categories and their module categories,  Hopf monads on tensor categories and
their relation with the notion of exact sequences of tensor categories.
In Section \ref{act-gdg} we discuss the main ingredients in our construction,
namely, group actions on $k$-linear abelian categories and the
related equivariantization process on one side, and group gradings on tensor
categories on the other side. In Section \ref{main-const} we define crossed
actions of matched pairs on tensor categories. In Section \ref{the-cat} we
present the main construction of the paper, that is, we prove here that every
crossed action gives rise to a tensor category. The main general properties of
this tensor category are studied in Section \ref{ppties}. In Section
\ref{braidings} we introduce $(G, \Gamma)$-crossed braidings and prove that a
$(G, \Gamma)$-crossed tensor category equipped with a $(G, \Gamma)$-crossed
braiding gives rise to a braided tensor category. In Section \ref{ejemplos} we
give examples of the main constructions from $G$-crossed categories and abelian
extensions of Hopf algebras.

\subsection*{Acknowledgement} This paper was partly written during a 
research stay in the University of Hamburg. The author thanks the Humboldt
Foundation, C. Schweigert and the Mathematics Department of U. Hamburg for the
kind hospitality.

\section{Preliminaries}\label{prels}

\subsection{Matched pairs of groups}\label{mpair} A matched pair of groups is
characterized by the existence of a group
$H$  endowed
with an \emph{exact factorization} into subgroups isomorphic to $G$ and
$\Gamma$, respectively. 
That is, $H$ is a group containing subgroups $\tilde G\cong G$ and $\tilde
\Gamma
\cong \Gamma$, such that 
$$H = \tilde G \, \tilde \Gamma, \qquad \tilde \Gamma \cap \tilde G = \{ e\}.$$

In fact, if $(G, \Gamma)$ is a matched pair, then there is a group structure,
denoted $G \Join \Gamma$ in the cartesian product $G \times \Gamma$, defined by
$$(g, s) (h, t) = (g(s \rhd h), (s\lhd h) t),$$ for all $g, h \in G$, $s, t\in
\Gamma$.
Conversely, given such a group $H$, we may identify $G$ and $\Gamma$ with
subgroups of $H$. In this way the
actions $\lhd:\Gamma \times G \to \Gamma$ and $\rhd:\Gamma \times G \to G$ are
determined by the relations $$sg = (s \vartriangleright
g)(s \vartriangleleft g),$$ for all $g \in G$, $s \in \Gamma$.

\medbreak  Let $(G, \Gamma)$ be a matched pair of groups.  Relations
\eqref{matched} imply that $s \vartriangleright e =
e$ and $e \vartriangleleft g = e$, for all $s \in \Gamma$, $g \in    G$.

\medbreak Using relations
\eqref{matched} it is also not difficult to show that the following conditions
are
equivalent:
  \begin{enumerate}
   \item[(i)] The action $\vartriangleleft:  \Gamma
\times G \longrightarrow \Gamma$ is trivial.
   \item[(ii)]   The action $\vartriangleright : \Gamma \times G
\longrightarrow G$ is by group automorphisms.
  \end{enumerate}
If these conditions hold, then the group $G \Join \Gamma$ coincides with the
semidirect
product $G \rtimes \Gamma$.

Similarly, the conditions
  \begin{enumerate}
   \item[(i')] The action $\vartriangleright: \Gamma \times G
\longrightarrow G$ is trivial.
   \item[(ii')]   The action $\vartriangleleft: 
\Gamma \times G \longrightarrow \Gamma$ is by group automorphisms.
  \end{enumerate}
are equivalent and, if they hold, then the group $G \Join \Gamma$ coincides with
the
semidirect product $G \ltimes \Gamma$.

\subsection{Tensor categories}\label{tc} A \emph{tensor category} over $k$ is a
$k$-linear
abelian rigid mo\-noidal category $\C$ such that the tensor product $\otimes: \C
\times \C \to \C$ is $k$-bilinear and the following conditions are satisfied:
\begin{itemize}
\item $\C$ is \emph{locally finite}, that is,
every object of $\C$ has finite length and Hom spaces are finite dimensional. 
\item The unit object $\uno \in \C$ is simple.
\end{itemize}
 Note that since $k$ is algebraically closed, then an object $X$ of $\C$ is
simple if and only if it is scalar, that
is, if and only if $\End(X) \cong k$. 

If $\C$ is a tensor category over $k$, then the functor $\otimes: \C \times \C
\to \C$ is exact in both variables.

A \emph{tensor subcategory} of a tensor category $\C$ is a full 
subcategory $\D$ of $\C$ which is closed under the operations of taking tensor
products, subobjects and dual objects (so in particular $\uno \in \D$).    A
tensor subcategory is
itself a tensor category with tensor product inherited from that of $\C$.

\medbreak A \emph{finite tensor category} over $k$ is a tensor category $\C$
over $k$ which satisfies either of the following equivalent conditions:
\begin{itemize}
\item $\C$ has enough projective objects and finitely many simple objects.
\item $\C$ has a projective generator, that is, an object $P\in \C$ such that
the functor $\Hom_\C(P, -)$ is faithful exact.
\item $\C$ is equivalent as a $k$-linear category to the category of finite
dimensional representations of a finite dimensional $k$-algebra.
\end{itemize}

A \emph{fusion category} over $k$ is a semisimple finite tensor category over
$k$. 

Let $G$ be a finite group. The category of finite dimensional
representations of $G$ over $k$ will be denoted by $\Rep
G$. This is a finite tensor category over $k$; it is a fusion category if and
only if the characteristic of $k$ does not divide the order of $G$.

All tensor categories in this paper will be assumed to be strict.

\medbreak Let $\C, \D$ be tensor categories over $k$. A  $k$-linear exact strong
monoidal
functor $F: \C \to \D$  will be called a \emph{tensor
functor}. Such functor is automatically faithful.

\medbreak A \emph{braided tensor category} over $k$ is a tensor category $\C$
endowed with a \emph{braiding}, that is, a natural isomorphism $\sigma: \otimes
\to \otimes^{\operatorname{op}}$ satisfying the following \emph{hexagon
conditions}: 
$$\sigma_{X, Y\otimes Z} = (\id_Y\otimes \sigma_{X, Z}) \, (\sigma_{X, Y}
\otimes \id_Z), \quad 
\sigma_{X\otimes Y, Z} = (\sigma_{X, Z} \otimes \id_Y) \, (\id_X\otimes
\sigma_{Y, Z}) ,$$
for all $X, Y, Z \in \C$.

\medbreak A (left) \emph{module category} over a tensor
category $\C$ is a locally finite $k$-linear abelian category $\M$ endowed with
a bifunctor $\otb: \C \times \M \to \M$, which is $k$-bilinear and exact, and
satisfies 
natural associativity and unit conditions. 

A module category $\M$ is called
\emph{indecomposable} if it is not equivalent to a direct sum of
two nonzero module categories.
It is called
\emph{exact} if for every projective object $P \in \C$ and for every object $M
\in \M$,  $P \otb M$ is a projective object of $\M$. See \cite{EO}.

It follows from \cite[Proposition 2.1]{EO} that every tensor category $\C$ is an
exact indecomposable module category over any tensor subcategory $\D$ with
respect to the action $\D \times \C \to \C$ given by the tensor product of $\C$.

As a consequence of this fact, we obtain that a finite tensor category  $\C$ is
a fusion
category if and only if its unit object $\uno$ is  projective.

\subsection{Hopf monads on tensor categories}\label{hmnds} 

Let  $\C$ be a tensor category over $k$. Recall that a \emph{monad} on $\C$ is
an endofunctor $T$ of $\C$ endowed with natural transformations $\mu: T^2 \to T$
and $\eta:\id_\C \to T$ called, respectively, the multiplication and unit of $T$
such that 
\begin{equation}
\mu_XT(\mu_X) = \mu_X\mu_{T(X)}, \quad \mu_X\eta_{T(X)} = \id_{T(X)} =
\mu_XT(\eta_X),
\end{equation}
for all objects $X \in \C$. 

\medbreak The monad $T$ is a \emph{bimonad} if it is a comonoidal endofunctor of
$\C$ such that the product $\mu$ and the unit $\eta$ are comonoidal
transformations. That is, if the comonoidal structure of $T$ is given by natural
transformations
$$T_2(X,Y): T(X \otimes Y) \to T(X) \otimes T(Y), $$  $X, Y \in \C$ and $T_0:
T(\uno) \to \uno$, then, for all objects $X, Y \in \C$, we have 
\begin{align}
\label{t2-mu}&T_2(X,Y)\mu_{X\otimes Y} = (\mu_X \otimes \mu_Y)
T_2(T(X),T(Y))T(T_2(X,Y)),\\
\label{t2-eta}& T_0 \mu_\uno = T_0T(T_0), \quad T_2(X,Y) \eta_{X\otimes
Y}=\eta_X
\otimes \eta_Y,\quad T_0 \eta_\uno =\id_\uno.
\end{align} 

\medbreak A bimonad $T$ is called a \emph{Hopf monad} provided that the fusion
operators $H^l:
T(\id_\C \otimes T) \to T\otimes T$ and $H^r: T(T \otimes \id_\C) \to T\otimes
T$ defined, for every $X, Y \in \C$, by
\begin{align*} & H^l_{X, Y}: =  (\id_{T(X)} \otimes \mu_Y) \, T_2(X, T(Y)):
T(X\otimes T(Y)) \to T(X) \otimes T(Y),\\
& H^r_{X, Y}: = (\mu_X \otimes \id_{T(Y)}) \, T_2(T(X), Y): T(T(X) \otimes Y)
\to T(X) \otimes T(Y), \end{align*}
are isomorphisms. 

\medbreak Let $T$ be a  $k$-linear right exact Hopf monad on $\C$. Then the
category $\C^T$ of $T$-modules in $\C$ is a tensor category over $k$. 

Recall that the objects of $\C^T$ are pairs $(X, r)$, where $X$ is an object of
$\C$ and $r:T(X) \to X$ is a morphism
in $\C$, such that
\begin{equation*}
r T(r)= r \mu_X, \quad \quad r \eta_X= \id_X.
\end{equation*}
If $(X,r), (X',r') \in \C^T$, a morphism $f: (X,r)\to  (X',r')$ is a morphism
$f:X \to X'$ in $\C$ such that $f r = r' T(f)$.

The tensor product of two objects $(X,r), (X',r') \in \C^T$ is defined by 
\begin{equation}\label{tp-ct}(X,r)\otimes (X', r') = (X\otimes X', (r\otimes
r')T_2(X,X')),
\end{equation} and the  unit  object of $\C^T$ is  $(\uno, T_0)$. See \cite{bv},
\cite{blv}, \cite[Proposition 2.3]{tensor-exact}.

\medbreak Moreover, in this situation, the forgetful functor $F: \C^T \to \C$, 
$F(X, r) = X$, is a strict tensor functor. The functor $F$ is dominant if and
only if
the Hopf monad $T$ is faithful.

\medbreak A \emph{quasitriangular Hopf monad} on a tensor category $\C$ is a
Hopf monad $T$ equipped with an \emph{$R$-matrix} $R$, that is, $R$ is a
$*$-invertible natural transformation $$R_{X, Y}: X \otimes Y \to T(Y) \otimes
T(X), \quad X, Y \in \C,$$
satisfying the following conditions, for all objects $X, Y, Z \in \C$:
\begin{align}\label{r1}& (\mu_X \otimes \mu_Y) R_{TX, TY} T_2(X, Y) = (\mu_X
\otimes \mu_Y) T_2(TY, TX) T(R_{X, Y}),\\
\label{r2}&(\id_{TZ} \otimes T_2(X, Y)) R_{X\otimes Y, Z} = (\mu_Z \otimes
\id_{TX\otimes TY}) (R_{X, TZ} \otimes \id_{TY}) (\id_X\otimes R_{Y, Z}),\\
\label{r3}& (T_2(Y ,Z) \otimes \id_{TX}) R_{X, Y\otimes Z} = (\id_{TY\otimes TZ}
\otimes \mu_X) (\id_{TY} \otimes R_{TX, Z}) (R_{X, Y} \otimes \id_Z).
\end{align} The $*$-invertibility of $R$ means that the natural morphisms
$$R^{\#}_{(X, r), (Y, s)} = (s \otimes r) R_{X, Y}: X \otimes Y \to Y \otimes
X,$$ are isomorphisms, for all objects $(X, r), (Y, s) \in \C^T$.
See \cite[Subsection 8.2]{bv}.

In view of \cite[Theorem 8.5]{bv}, if $T$ is a quasitriangular Hopf monad on
$\C$, then $\C^T$ is a braided tensor category with braiding $\sigma_{(X, r),
(Y, s)} : (X, r) \otimes (Y, s) \to (Y, s) \otimes (X, r)$, defined in the form
$\sigma_{(X, r), (Y, s)} = (s \otimes r) R_{X, Y}$.

\subsection{Exact sequences of tensor categories}

Let $\C$, $\C''$ be tensor
categories over $k$. A tensor functor $F: \C \to \C''$ is called \emph{normal}
if every object $X$ of $\C$, there
exists a subobject $X_0 \subset X$ such that $F(X_0)$ is the
largest trivial subobject of $F(X)$.

If the functor $F$ a right adjoint $R$, then $F$ is normal if and only if
$R(\uno)$ is a trivial object of $\C$ \cite[Proposition 3.5]{tensor-exact}.

For a tensor functor $F: \C \to \C''$, let $\KER_F$ denote the tensor
subcategory
$F^{-1}(\langle \uno \rangle) \subseteq \C$  of objects $X$ of
$\C$ such that $F(X)$ is a trivial object of $\C''$.

\medbreak Let $\C', \C, \C''$ be tensor categories over $k$. An \emph{exact
sequence of tensor categories} is a sequence of
tensor functors
\begin{equation}\label{exacta-fusion}\xymatrix{\C' \ar[r]^f & \C \ar[r]^F &
\C''}
\end{equation}
such that the tensor functor $F$ is dominant and normal and the tensor functor
$f$ is a full embedding whose  essential image  is $\KER_F$.
See \cite{tensor-exact}. 

\medbreak 
The \emph{induced Hopf algebra} $H$ of the exact sequence
\eqref{exacta-fusion} is defined as
the coend of the fiber functor $\omega_F = \Hom_{\C''}(\uno, Ff): \C' \to
\vect_k$.   There is
an equivalence of tensor categories $\C' \simeq \corep H$.    See
\cite[Subsection 3.3]{tensor-exact}.

\medbreak By
\cite[Theorem 5.8]{tensor-exact} exact sequences \eqref{exacta-fusion} with
finite dimensional induced Hopf algebra $H$ are classified by normal faithful
right exact $k$-linear Hopf
monads $T$ on $\C''$, such that the Hopf monad of the restriction of $T$
to the trivial subcategory of $\C''$ is isomorphic to $H$.
Recall that a $k$-linear right exact Hopf monad $T$ on a tensor
category $\C''$ is called \emph{normal} if $T(\uno)$ is a trivial object of
$\C''$. 

\section{Group actions and group gradings on $k$-linear and tensor
categories}\label{act-gdg}
In this section we discuss some facts on group actions and group gradings on
$k$-linear and tensor categories that will be used later on.

\subsection{Group actions on $k$-linear abelian categories} Let $G$
be a group and let $\C$ be a $k$-linear abelian category. 

Let $\underline G$ be the strict monoidal category whose objects are the
elements of $G$ and morphisms are identities, with tensor product defined as the
multiplication in $G$ and unit object $e \in G$. 
Let also $\underline \Aut \, \C$ be the strict monoidal category whose objects
are $k$-linear autoequivalences of $\C$, morphisms are natural transformations,
with tensor product defined by composition of endofunctors and natural
transformations and unit object $\id_\C$.

\medbreak Consider the strict monoidal category $\underline{G}^{\op}$ obtained
from $\underline G$ by reversing the tensor product. That is, the underlying
category of $\underline{G}^{\op}$ is $\underline G$, while the tensor product in
$\underline{G}^{\op}$ is defined by $g \otimes h = hg$, $g, h \in G$.

\medbreak By a \emph{right action} of $G$ on $\C$  by $k$-linear
autoequivalences we shall understand a monoidal functor $\rho:
\underline{G}^{\op} \to \underline \Aut \, \C$. 
That is, for every $g \in G$, we have a
$k$-linear functor $\rho^g: \C \to \C$ and natural isomorphisms
$$\rho^{g,h}_2 : \rho^g \rho^h \to \rho^{hg}, \quad g, h \in G,$$ and
$\rho_0 : \id_\C \to \rho^e$, satisfying
\begin{align}\label{rro-2} & (\rho^{ba, c}_2)_X \, (\rho^{a,b}_2)_{\rho^c(X)} =
(\rho^{a, cb}_2)_X \, \rho^a((\rho^{b,c}_2)_X), \\ \label{rro-3}& (\rho^{a,
e}_2)_X \rho^{a}({\rho_0}_X) = \id_{\rho^a(X)} = (\rho^{e, a}_2)_X
(\rho_0)_{\rho^{a}(X)},
\end{align}
for all $X \in \C$, $a, b, c \in G$.

\subsection{Equivariantization}\label{equiv-ab}

Let $\rho: \underline G^{\op} \to \underline \Aut \, \C$ be a right action of
$G$ on $\C$ by $k$-linear autoequivalences. A
$G$-equivariant object is a pair $(X, r)$, where $X$ is an object
of $\C$ and $r = (r^g)_{g \in G}$ is a collection of isomorphisms $r^g:\rho^gX
\to X$, $g \in G$, satisfying
\begin{equation}\label{deltau} r^g \rho^g(r^h) = r^{hg} (\rho^{g,
h}_2)_X, \quad \forall g, h \in G, \quad \quad
r^e{\rho_0}_X=\id_X.\end{equation} A $G$-equivariant morphism $f: (X, r) \to
(Y, r')$ is a morphism $f: X \to Y$ in $\C$ such
that $fr^g = {r'}^g\rho^g(f)$, for all $g \in G$.

The category of $G$-equivariant objects and morphisms is a $k$-linear abelian
category, denoted $\C^G$, called the  \emph{equivariantization} of $\C$ under
the action $\rho$.

\medbreak Suppose that $G$ is a finite group. Let $T^\rho: \C \to \C$ be the
endofunctor of $\C$ defined by $T^\rho =
\bigoplus_{g \in G}\rho^g$. Then $T^\rho$ is a $k$-linear exact monad
on $\C$ with multiplication $\mu:{T^\rho}^2 = \bigoplus_{g, h \in G}\rho^g
\rho^h \to \bigoplus_{g \in G}\rho^g = T^\rho$, given componentwise by the
isomorphisms $\rho_2^{g, h}: \rho^g \rho^h \to \rho^{hg}$, and unit $\eta =
\rho_0:\id_\C \to \rho^e \to T^\rho$.

Since the unit $\eta$ of $T^\rho$ is a monomorphism, then $T^\rho$ is a faithful
endofunctor of $\C$ \cite[Lemma 2.1]{tensor-exact}.

Extending the terminology of \cite{tensor-exact}, we shall call $T^\rho$ the
monad of the group action $\rho$. (Note however, that the group actions
considered \textit{loc. cit.} are by tensor autoequivalences on tensor
categories.)

The canonical isomorphisms $$\Hom_\C(\bigoplus_{g \in G}\rho^g(X), X) \cong
\prod_{g \in G} \Hom_\C(\rho^g(X), X),$$ $X \in \C$, induce an equivalence of
categories over $\C$ between the category $\C^{T^\rho}$ of $T^\rho$-modules in
$\C$ and the equivariantization $\C^G$.
See \cite[Subsection 5.3]{tensor-exact}.

\begin{remark} Suppose that $\C$ is a tensor category. Assume in addition that
the action of
$G$ is given by tensor autoequivalences of $\C$, that is, the endofunctor
$\rho^g$ is a tensor functor, for all $g \in G$, and $\rho_2^{g, h}:\rho^g\rho^h
\to \rho^{hg}$, $\rho_0: \id_\C \to \rho^e$ are natural isomorphisms of monoidal
functors. 

Then $T^\rho$ is Hopf monad on $\C$ with comonoidal
structure $$T_2(X, Y): \bigoplus_{g \in G}\rho^g(X \otimes Y) \to \bigoplus_{g,
g' \in G}\rho^g(X) \otimes \rho^g(Y),$$ and $T_0: \bigoplus_{g \in
G}\rho^g(\uno)
\to \uno$, given componentwise by the monoidal structure $\rho_2^g: \rho^g \circ
\otimes \to \rho^g \otimes \rho^g$ and $\rho^g_0:\rho^g(\uno) \to \uno$ of the
functors $\rho^g$, $g \in G$. 

Thus the equivariantization $\C^G$ is a tensor category with tensor product
defined by the formula
\begin{equation}\label{tens-equiv}(X, r) \otimes (X', r') = (X \otimes X', (r
\otimes r')(\rho_2)_{X, X'}). \end{equation} 
In addition, if $\C$ is a finite tensor category, then so is $\C^G$. 
If $\C$ is a fusion category and the characteristic of $k$ does not
divide the order of $G$, then $\C^G$ is also a fusion category.

Furthermore, $T^\rho$ is a normal cocommutative Hopf monad on $\C$ and 
the forgetful functor $F:\C^G \to \C$ gives rise to a central exact sequence of
tensor categories
\begin{equation}\Rep G \toto \C^G \toto \C. \end{equation}
See \cite[Corollary 2.22]{tensor-exact}, \cite[Example 2.5]{indp-exact}.
\end{remark}

\subsection{Group gradings on tensor categories}

Let $G$ be a group and let $\C$ be a tensor category over $k$. Let 
\begin{equation}\label{grading}
 \C = \bigoplus_{g \in G}\C_g,
\end{equation}
be a \emph{$G$-grading} on
$\C$. That is, for every $g \in G$, $\C_g$ is a full subcategory of $\C$ and the
following conditions hold:

\begin{itemize}
 \item For every object $X$ of $\C$ we
have a decomposition $X \cong \bigoplus_{g \in G}X_g$, where $X_g \in \C_g$, for
all $g \in G$.
\item For all $X \in \C_g$, $Y \in \C_h$, $g \neq h \in G$, we have $\Hom_\C(X,
Y) = 0$.
\item $\C_g \otimes \C_h \subseteq
\C_{gh}$, for all $g, h \in G$. 
\end{itemize}

The subcategories $\C_g$, $g \in G$, are called the \emph{homogeneous
components} of the grading. A $G$-grading \eqref{grading} is called
\emph{faithful} if $\C_g \neq 0$, for
all $g \in G$.

\medbreak Let $\C = \bigoplus_{g \in G}\C_g$ and $\D= \bigoplus_{g \in G}\D_g$
be
$G$-graded $k$-linear abelian categories. 
A functor $F:\C \to \D$ will be called a \emph{$G$-graded functor} if
$F(\C_g) \subseteq \D_g$, for all $g \in G$.

\begin{lemma}\label{f-gr} Let $F:\C \to \D$ be a $G$-graded functor between
$G$-graded $k$-linear abelian categories categories $\C$, $\D$. Suppose $F$ is
dominant. Then, for all $g \in G$, $F$ induces by restriction a dominant functor
$F: \C_g \to \D_g$. \end{lemma}

\begin{proof} Let $g \in G$ and let $Y$ be any object of $\D_g$. Since $F$ is
dominant, there exists $X \in \C$ such that $Y$ is a subobject of $F(X)$. Let $X
\cong \bigoplus_{h \in G}X_h$ be decomposition of $X$ into a direct sum of
homogeneous objects $X_h \in  \C_h$. Since $F$ is a $G$-graded functor, $F(X)
\cong \bigoplus_{h \in G}F(X_h)$ is a decomposition of $F(X)$ into a direct sum
of homogeneous objects $F(X_h) \in  \D_h$. Then $Y$ must be a subobject of
$F(X_g)$, because $\Hom_\D(Y, F(X_h)) = 0$, for all $h \neq g$. This proves the
lemma.
\end{proof}

\begin{lemma}\label{uno-dual} Let $\C$ be a $G$-graded tensor category.   Then
we have:
\begin{enumerate}
 \item[(i)] The unit object of $\C$ belongs to $\C_e$.
 \item[(ii)]     Suppose $Y \in \C_g$
is a nonzero homogeneous object. Then $Y^*$ and $^*Y$ belong to $\C_{g^{-1}}$.
\end{enumerate}
\end{lemma}

\begin{proof} Since the unit object $\1$ is simple, then we get that $\uno \in
\D$.  This shows (i). To show (ii), let $h \neq g^{-1} \in G$ and let $X \in
\C_h$.   We have an isomorphism
$$\Hom_\C(X, Y^*) = \Hom_\C(X, \uno \otimes Y^*) \cong \Hom_\C(X \otimes
Y, \uno),$$ by \cite[Lemma 2.1.6]{BK}.    Hence $\Hom_\C(X, Y^*) = 0$
because $X \otimes Y \in \C_{hg}$ and $hg \neq e$. Similarly one can see that 
$\Hom_\C(X, ^*Y)
= 0$.   Therefore $Y^*, ^*Y \in \C_{g^{-1}}$, as claimed. This finishes the
proof of the lemma.
\end{proof}

\begin{proposition} Let $\C$ be a $G$-graded tensor category. Then the neutral
homogeneous component $\D = \C_e$ is a tensor subcategory of $\C$.
Besides, every homogeneous component $\C_g$,
$g \in G$, is an indecomposable exact left (and right) module category with the
action given by the tensor product of $\C$.
\end{proposition}

\begin{proof} It follows from Lemma \ref{uno-dual} that $\D$ contains the unit
object and is closed under the operations of taking duals. This implies that
$\D$ is a tensor subcategory of $\C$.

Let now $g \in    G$. By the definition of a $G$-grading we have  $\D
\otimes \C_g \subseteq \C_g$, and  $\C_g \otimes \D \subseteq \C_g$.  Therefore
$\C_g$ is  both a left and right module subcategory of $\C$ over the tensor
subcategory $\D$. Since $\C$ is an exact module category over $\D$, then so is
$\C_g$.

It remains to prove the indecomposability of $\C_g$. We may assume that $\C_g
\neq 0$. Let $X, Y$ be any nonzero objects of $\C_g$.  It follows from Lemma
\ref{uno-dual}  (ii) that $Z = Y^*  \otimes  X \in \D$.   

Observe that the functor $- \otimes X: \C \to \C$ is faithful exact. Then, since
by rigidity $\Hom_\C(\uno, Y \otimes Y^*) \neq 0$,  we obtain that $\Hom_\C(X, Y
\otimes Z) = \Hom_\C(X, Y \otimes Y^* \otimes X) \neq 0$. 
This implies that $\C_g$ is indecomposable as a right module category over
$\D$. 
Indecomposability as a left module category is shown similarly. This finishes
the proof of the proposition.
\end{proof}

\begin{corollary}\label{d-finite-fusion} Suppose $G$ is a finite group. Let $\C$
be a $G$-graded tensor
category with neutral homogeneous component $\D$.  Then  $\C$ is a finite tensor
category
(respectively, a fusion category) if and only if so is $\D$.
\end{corollary}

\begin{proof} Any tensor subcategory of a finite tensor category (respectively,
of a fusion category) is itself a finite tensor category  (respectively, a
fusion category). Then we only need to show the 'if' direction.  

Suppose first that $\D$ is a finite tensor category. Let $P\in \D$ be a
projective generator, that is $P$ is an object of $\D$ such that the functor
$\Hom_\D(P, -)$ is faithful exact. Since the group $G$ is finite, it will be
enough to show that every homogeneous component is a finite $k$-linear abelian
category.

Let $g \in G$ such that $\C_g \neq 0$. Note that since $\C$ is a tensor
category, then $\C_g$ is locally finite  (that is, it has finite
dimensional hom spaces and every object has finite length). Therefore it will be
enough to show that $\C_g$ has a projective generator. Let $X_0 \in \C_g$ be any
nonzero object.

By exactness of the left $\D$-module category $\C_g$, $P \otimes X_0$ is a
projective object of $\C_g$. Hence the functor $\Hom_{\C_g}(P \otimes X_0, -) =
\Hom_{\C}(P \otimes X_0, X)$ is exact. In addition, using the rigidity of $\C$,
we get for all $X \in \C_g$ a natural isomorphism
$$\Hom_{\C}(P \otimes X_0, X) \cong \Hom_{\C}(P
, X \otimes X_0^*) = \Hom_\D(P, X \otimes X_0^*).$$   Since both functors
$\Hom_\D(P, -)$ and $-\otimes X_0^*$ are faithful, then 
$\Hom_{\C_g}(P \otimes X_0, -)$ is faithful. Hence $P \otimes X_0$ is a
projective generator of $\C_g$.  Thus we obtain that $\C$ is a finite tensor
category.

\medbreak Suppose next that $\D$ is a fusion category. In particular it is
a finite tensor category and hence so is $\C$, by the previous part.
Since $\C$ is an exact module category over $\D$, then $\C$ is semisimple (see
\cite[Example 3.3]{EO}).   Hence $\C$ is also a fusion category, as claimed. 
\end{proof}

\begin{remark} Suppose $k$ is of characteristic zero. Let $\C$ be a fusion
category over $k$.   
Group gradings on $\C$ were classified in \cite{ENO3}.

By \cite[Proposition
2.9]{ENO2}, $\C$ admits a faithful $G$-grading if and only if its Drinfeld
center $\Z(\C)$ contains a Tannakian subcategory $\mathcal E$ such that
$\mathcal E \cong \Rep G$ as symmetric categories and $\mathcal E$ is contained
in the kernel of the forgetful functor $U : \Z(\C) \to \C$.

Observe that endowing $\C$ with a $G$-grading is equivalent to providing a map
$\partial: \Irr(\C) \to G$ such that $\partial(Z) = \partial(X) \partial(Y)$,
for all simple objects $X$, $Y$ and $Z$ of $\C$ such that $\Hom_\C(Z, X \otimes
Y) \neq 0$. 

The grading corresponding to a Tannakian subcategory $\E \cong \Rep
G$ of the center of $\C$ is defined as follows. Let $X$ be an object of $\C$ and
let $(V, \sigma) \in \E$. Then $U(V, \sigma) = V$ is a trivial object of $\C$
and thus it is equipped with a trivial half-brading $\tau_{X, V}: X \otimes V
\to V \otimes X$. Composing with the braiding $\sigma$ this gives an isomorphism
$\tau_{X, V}\sigma_{V, X}:V \otimes X \to V \otimes X$.

Let $X$ be a simple object of $\C$. Using that $\sigma$ is a braiding, we obtain
in this way a natural automorphism of tensor functors $U\vert_{\E} \to
U\vert_\E$. This is the same as an element $g \in G$, since $U\vert_\E$ is a
fiber functor on $\E$. This defines a map $\partial :\Irr(\C) \to G$, which is
seen to be $G$-grading using the hexagon axiom for the braiding of the center.

\medbreak  It follows from \cite[Proposition 8.20]{ENO}
that if $\C$ is a fusion category endowed with a faithful $G$-grading, then
$\FPdim \C = \vert G\vert \FPdim \D$.
\end{remark}

\section{$(G, \Gamma)$-crossed actions on tensor categories}\label{main-const}

Let $\C$ be a tensor
category over $k$ and let $(G, \Gamma)$ be a matched pair of  groups. 

\begin{definition}\label{crossed-action} A $(G, \Gamma)$-\emph{crossed action}
on the tensor category $\C$ consists of the following data:

\begin{itemize}\item  A $\Gamma$-grading on $\C$: $\C = \bigoplus_{s \in \Gamma}
\C_s$.

\item  A right action of $G$ on $\C$ by $k$-linear autoequivalences $\rho:
\underline{G}^{\op} \to \underline{\Aut}(\C)$ such that
\begin{equation}\label{rho-partial} \rho^g(\C_s) = \C_{s \lhd g},\quad \forall
g\in G, \, s\in \Gamma,\end{equation}

\item A collection of natural isomorphisms $\gamma = (\gamma^g)_{g\in G}$:
\begin{equation}\label{gamma}\gamma^g_{X, Y}: \rho^g(X \otimes Y) \to
\rho^{t\rhd g}(X) \otimes \rho^g(Y), \quad X \in \C, \, t\in \Gamma,\,  Y \in
\C_t, \end{equation}

\item A collection of isomorphisms $\gamma^g_0: \rho^g(\uno) \to \uno$, $g \in
G$.
\end{itemize}

These data are subject to the commutativity of the following diagrams:

\begin{itemize}\item[(a)] For all $g \in G$, $X \in \C$, $s, t \in \Gamma$, $Y
\in \C_s$, $Z
\in \C_t$,
\begin{equation*}
\xymatrix @C=0.6in @R=0.45in{
\rho^g(X \otimes Y \otimes Z) \ar[rr]^{\gamma^g_{X\otimes Y, Z}}
\ar[d]_{\gamma^g_{X, Y \otimes Z}} & & \rho^{t \rhd g}(X \otimes Y) \otimes
\rho^g(Z) \ar[d]^{\gamma^{t\rhd g}_{X, Y} \otimes \id_{\rho^g(Z)}} \\
\rho^{st \rhd g}(X) \otimes \rho^g(Y \otimes Z)\ar[rr]_{\id_{\rho^{st \rhd
g}(X)}
\otimes \gamma^g_{Y, Z} \qquad }&& \rho^{s \rhd (t \rhd g)}(X) \otimes \rho^{t
\rhd g}(Y) \otimes \rho^g(Z)}
\end{equation*}

\item[(b)] For all $g \in G$, $X \in \C$,
\begin{equation*}\xymatrix @C=0.6in @R=0.45in{
\rho^g(X) \otimes \rho^g(\uno)  \ar[dr]_{\id_{\rho^g(X)} \otimes \gamma^g_0
\quad } & \ar[l]_{\quad \gamma^g_{X, \uno}} \rho^g(X) \ar[d]^=
\ar[r]^{\gamma^g_{\uno, X} \quad } & \rho^g(\uno) \otimes \rho^g(X)
\ar[dl]^{\gamma^g_0 \otimes \id_{\rho^g(X)}} \\
& \rho^g(X) }\end{equation*}

\item[(c)] For all $g, h \in G$, $X \in \C$, $s\in \Gamma$, $Y \in \C_s$,
\begin{equation*}\xymatrix@C=0.6in @R=0.45in{
\rho^g\rho^h(X\otimes Y) \ar[dd]_{\rho^g(\gamma^h_{X, Y})} \ar[r]^{{\rho_2}_{X
\otimes Y}^{g, h}} & \rho^{hg}(X \otimes Y) \ar[d]^{\gamma^{hg}_{X, Y}} \\
& \rho^{s \rhd hg}(X) \otimes \rho^{hg}(Y) \\
\rho^g(\rho^{s\rhd h}(X) \otimes \rho^h(Y)) \ar[r]_{\gamma^g_{\rho^{s\rhd h}(X),
\rho^h(Y)}}  & \rho^{(s\lhd h)\rhd g}\rho^{s\rhd h}(X) \otimes \rho^g\rho^h(Y)
\ar[u]_{{\rho_2}_X^{(s\lhd h) \rhd g, s \rhd h} \otimes
{\rho_2}_Y^{g, h}} }\end{equation*}

\item[(d)] For all $g, h \in G$, 
\begin{equation*}\xymatrix@C=0.6in @R=0.45in{
\rho^g\rho^h(\uno) \ar[d]_{\rho^{g}(\gamma^h_0)} \ar[r]^{(\rho_2^{g, h})_{\uno}}
& \rho^{hg}(\uno) \ar[d]^{\gamma^{hg}_0} \\
\rho^g(\uno) \ar[r]_{\gamma^{g}_0} & \uno}
\end{equation*}

\item[(e)] For all $X \in \C$, $s \in \Gamma$, $Y \in \C_s$, 
\begin{equation*}\xymatrix@C=0.6in @R=0.45in{
X\otimes Y \ar[dr]_{{\rho_0}_X \otimes {\rho_0}_Y} \ar[r]^{{\rho_0}_{X\otimes
Y}} & \rho^e(X\otimes Y) \ar[d]^{\gamma^{e}_{X, Y}} \\
& \rho^e(X)\otimes \rho^e(Y)}
\xymatrix@C=0.6in @R=0.45in{
\uno \ar[dr]_= \ar[r]^{{\rho_0}_\uno} & \rho^e(X\otimes Y) \ar[d]^{\gamma^e_0}
\\
& \uno}
\end{equation*}
\end{itemize}

We shall say that $\C$ is a \emph{$(G, \Gamma)$-crossed tensor category} if it
is endowed with a $(G, \Gamma)$-crossed action.
\end{definition}

\begin{remark} Recall that $s \rhd e = e$, for all $s \in \Gamma$. Thus
conditions (a) and (b) in the definition of a $(G, \Gamma)$-crossed tensor
category imply that $\rho^e: \C \to \C$ is a monoidal functor with monoidal
structure $\gamma^e_{X, Y}: \rho^e(X \otimes Y) \to \rho^e(X) \otimes
\rho^e(Y)$, $X, Y \in \C$, and $\gamma^e_0: \rho^e(\uno) \to \uno$.

Commutativity of the diagrams in condition (e) amounts to the requirement that
the natural isomorphism $\rho_0:\id_\C \to \rho^e$ is a monoidal isomorphism.
\end{remark}

\begin{remark}\label{suff-mp} Suppose that $G$ and $\Gamma$ are groups endowed
with mutual
actions by permutations $\Gamma \overset{\vartriangleleft}\longleftarrow \Gamma
\times G \overset{\vartriangleright}\longrightarrow G$. Let  $\C = \bigoplus_{s
\in \Gamma} \C_s$ be a $\Gamma$-graded tensor category over $k$ and let  $\rho:
\underline{G}^{\op} \to \underline{\Aut}(\C)$ be a right action of $G$ on $\C$
by $k$-linear autoequivalences satisfying \eqref{rho-partial}, and such that
there exists  a collection of natural isomorphisms $\gamma = (\gamma^g)_{g\in
G}$ as in \eqref{gamma}, satisfying condition (c). 

Assume that the $\Gamma$-grading on $\C$ is faithful and that the action $\rho:
\underline{G}^{\op} \to \underline{\Aut}(\C)$ is
faithful, that is, $\rho^g \cong \rho^e$ if and only if $g = e$. Then the
actions $\lhd$,
$\rhd$ make $(G, \Gamma)$ into a matched pair of groups.

Note that the faithfulness of $\rho$ holds for instance if 
the action  $\lhd: \Gamma \times G \longrightarrow \Gamma$ is faithful: in fact,
if $g \in G$ is such that $\rho^g \cong \rho^e$ then, since the $\Gamma$-grading
is faithful, we get from \eqref{rho-partial} that $s \lhd g = s$, for all $s \in
\Gamma$. Hence $g = e$.

\begin{proof} Let $s, t \in \Gamma$, $g \in G$. For all objects $X \in \C_t$, $Y
\in \C_s$, we have $X \otimes Y \in \C_{ts}$. Then, by \eqref{rho-partial},
$\rho^{g}(X \otimes Y) \in \C_{ts\lhd g}$. On the other hand, under the
isomorphism $\gamma^g$, $$\rho^{g}(X \otimes Y) \cong \rho^{s\rhd g}(X) \otimes
\rho^{g}(Y) \in \C_{t\lhd (s\rhd g)} \otimes \C_{s \lhd g} \subseteq \C_{(t\lhd
(s\rhd g))(s \lhd g)}.$$
Since, by assumption, the $\Gamma$-grading on $\C$ is faithful, we may take $X$
and $Y$ to be nonzero objects. Hence we obtain $$ts\lhd g = (t\lhd (s\rhd g))(s
\lhd g), \quad \textrm{ for all } s, t \in \Gamma, \quad g \in G.$$

Let now $g, h \in G$, $s \in \Gamma$, and let $X, Y$ be objects of
$\C$ such that $Y \in \C_s$ and $Y \neq 0$. The right hand side of (c)
defines a natural isomorphism 
$\rho^g\rho^h(X\otimes Y) \to \rho^{s \rhd hg}(X) \otimes \rho^{hg}(Y)$. On the
other hand, the left hand side of (c) defines a natural isomorphism
$$\rho^g\rho^h(X\otimes Y) \to \rho^{(s \rhd h)((s\lhd h) \rhd g)}(X) \otimes
\rho^{hg}(Y).$$
Since the tensor product of $\C$ is a faithful functor in each variable, we get
a natural isomorphism $\rho^{s \rhd hg} \cong \rho^{(s \rhd h)((s\lhd h) \rhd
g)}$. Because of faithfulness of the action $\rho$, we obtain
$$s \rhd hg = (s \rhd h)((s\lhd h) \rhd g),$$ for all $s\in \Gamma$, $g, h \in
G$.
Therefore $(G, \Gamma)$ is matched pair of groups, as claimed.
\end{proof}
\end{remark}

\section{The category $\C^{(G, \Gamma)}$}\label{the-cat}
Let $\C$ be a $(G, \Gamma)$-crossed tensor category. Since the group $G$
acts on $\C$ by $k$-linear autoequivalences, we may consider the
equivariantization $\C^G$, which is a $k$-linear abelian category. 

\medbreak Let $(X, r)$ be an equivariant object. That is, $r^g:\rho^g(X) \to X$
are isomorphisms, for all
$g \in G$, satisfying the relations \eqref{deltau}. Let $X \cong \bigoplus_{s
\in \Gamma}X_s$ be a decomposition of $X$ as a direct sum of homogeneous objects
$X_s \in \C_s$, $s \in \Gamma$. 

Condition \eqref{rho-partial} implies that, for
all $g \in G$, $s\in \Gamma$, $r^g$ induces by restriction an isomorphism
$r^g_s:\rho^g(X_s) \to X_{s \lhd g}$.

\begin{theorem}\label{tens-prod} Let $\C$ be a $(G, \Gamma)$-crossed tensor
category. Then the equivariantization of $\C$ under the action $\rho$ is a
tensor category over $k$ with tensor product defined as follows:
\begin{equation}\label{f-tensor}(X, r) \otimes (Y, r') = (X \otimes Y, \tilde
r),
\end{equation} and unit object $(\uno, (\rho^g_{\uno})_{g \in G})$,
where, for all $g \in G$, ${\tilde r}^g$ is defined as the composition
$$\bigoplus_{s \in \Gamma} \rho^g(X \otimes Y_s)
\overset{\oplus_s\gamma^g_{X, Y_s}}\toto \bigoplus_{s \in \Gamma} \rho^{s \rhd
g}(X) \otimes \rho^g(Y_s) \overset{\oplus_s r^{s \rhd g} \otimes {r'}^g_s}\toto
\bigoplus_{s \in \Gamma} X \otimes Y_{s\lhd g} = X \otimes Y,$$
for $Y = \oplus_{s \in \Gamma}Y_s$, $Y_s \in \C_s$.
\end{theorem}

Observe that the action of $G$ on $\C$ is not necessarily by tensor
autoequivalences. Therefore the equivariantization $\C^G$ is not a tensor
category with the tensor product defined by formula \eqref{tens-equiv}. The
tensor category in
Theorem  \ref{tens-prod} will be indicated by $\C^{(G, \Gamma)}$ to emphasize
this distinction.

\begin{proof}
Consider the endofunctor $T = \bigoplus_{g \in G}\rho^g$ of $\C$ defined by
the action of $G$. Then $T$ is a $k$-linear exact faithful endofunctor of $\C$.
Moreover, $T$ is a monad on $\C$ with multiplication $\mu: T^2 \to T$ and unit
$\eta: \id_\C \to T$ induced, respectively, by the morphisms
$\rho_2^{g, h}$, $g, h \in G$, and $\rho^e$. See Subsection  \ref{equiv-ab}.

The natural isomorphisms
$$\gamma^g_{X, Y}: \rho^g(X \otimes Y) \to \rho^{s\rhd g}(X) \otimes \rho^g(Y),
\quad g \in G, \, X \in \C, Y \in \C_s,$$
induce canonically a natural transformation $$T_2(X, Y)\!: T(X \otimes Y) =
\bigoplus_{g \in G}\rho^g(X \otimes Y)  \to \bigoplus_{g, h \in G} \rho^g(X)
\otimes \rho^h (Y) = T(X) \otimes T(Y),$$ $X,Y \in \C$.
Similarly, the morphisms $\gamma^g_0: \rho^g(\uno) \to \uno$ induce a morphism
$$T_0
: T(\uno) = \bigoplus_{g \in G}\rho^g(\uno) \to \uno.$$
Conditions (a) and (b) in Definition \ref{crossed-action}
imply that $T$ is a comonoidal endofunctor of $\C$ with comonoidal structure
given by $T_2$ and $T_0$.
Conditions
(c), (d)  and (e) imply that $\mu: T^2 \to T$ and $\eta: \id_\C
\to T$ are comonoidal transformations, that is, they satisfy the relations
\eqref{t2-mu} and \eqref{t2-eta}. Hence $T$ is a bimonad on $\C$.

\medbreak  We claim that $T$ is a Hopf monad on $\C$.   This will entail that
$\C^{(G, \Gamma)}$ is a tensor category with the prescribed structure since, by
the definition of the tensor product of $\C^{(G, \Gamma)}$ given in
\eqref{f-tensor}, it coincides with the one given by formula \eqref{tp-ct} for
the tensor product in the category $\C^T$ of $T$-modules in $\C$.

\medbreak According to the results in \cite[Section 2]{blv}, to establish the
claim it will be enough to show that the fusion operators $H^l$ and $H^r$ of $T$
are isomorphisms. Recall that $H^l$ and $H^r$ are defined, respectively, by
\begin{align*}& H^l_{X, Y} = (\id_{T(X)} \otimes \mu_Y) \, T_2(X, T(Y)): T(X
\otimes T(Y)) \to T(X) \otimes T(Y), 
\\ & H^r_{X, Y} = (\mu_X \otimes \id_{T(Y)}) \, T_2(T(X), Y) : T(T(X) \otimes Y)
\to T(X) \otimes T(Y).\end{align*}

\medbreak Let $X$ be any object of $\C$. For every homogenous object $Y \in
\C_s$, $s \in \Gamma$, the operator $$H^r_{X, Y}: 
\bigoplus_{g, h \in G} \rho^g (\rho^h(X) \otimes Y) \to \bigoplus_{g, h \in
G}\rho^g(X) \otimes \rho^h(Y),$$ is given componentwise by
the composition of isomorphisms
$$(\rho_2^{s\rhd g, h}\otimes \id_{\rho^g(Y)}) \, \gamma^g_{\rho^h(X), Y}:
\rho^g(\rho^h(X) \otimes Y) \to \rho^{h (s\rhd g)}(X) \otimes \rho^g(Y).$$
Since the map $G \times G \to G \times G$, $(g, h) \mapsto (h(s\rhd g), g)$, is
bijective for any $s \in \Gamma$, then $H^r_{X, Y}$ is an isomorphism. Therefore
$H^r_{X, Y}$ is an isomorphism for all $Y \in \C$.

\medbreak Similarly, if $X$ is any object of $\C$ and $Y \in \C_s$ is a
homogeneous object, $s \in \Gamma$, then $$H^l_{X, Y}: 
\bigoplus_{g, h \in G} \rho^g (X \otimes \rho^h(Y)) \to \bigoplus_{g, h \in
G}\rho^g(X) \otimes \rho^h(Y),$$ is given componentwise by
the composition of isomorphisms
$$(\id_{\rho^{(s \lhd h) \rhd g}(X)} \otimes \rho_2^{g, h}) \, \gamma^g_{X,
\rho^h(Y)}: \rho^g(X \otimes \rho^h(Y)) \to \rho^{(s \lhd h) \rhd g}(X) \otimes
\rho^{hg}(Y).$$
We conclude as before that $H^l_{X, Y}$ is an isomorphism for all $Y \in \C$.
Indeed, to see that for each $s \in \Gamma$ the map $G \times G \to G \times G$,
$(g, h) \mapsto ((s\lhd h)\rhd g, hg)$, is bijective, we argue as follows: the
composition of this map with the bijection $\id_G \times (s \rhd -): G \times G
\to G \times G$ gives the map
$(g, h) \mapsto ((s\lhd h)\rhd g, s \rhd hg)$. Using the compatibility condition
in \eqref{matched}, the last map is bijective with inverse $(u, v)\mapsto ((s
\lhd (s^{-1}\rhd vu^{-1}))^{-1} \rhd u, s^{-1} \rhd vu^{-1})$. Therefore $T$ is
a Hopf monad, and thus $\C^{(G, \Gamma)} = \C^T$ is a tensor category, as
claimed. \end{proof}

\section{Main properties}\label{ppties}

In this section we study the structure of the tensor category 
$\C^{(G, \Gamma)}$ arising from a $(G, \Gamma)$-crossed tensor
category for a general matched pair of finite groups $(G, \Gamma)$.

\begin{theorem}\label{exact-sequence} Let $\C$ be a $(G, \Gamma)$-crossed tensor
category and let $\C^{(G, \Gamma)}$ be the category defined by Theorem
\ref{tens-prod}. 
Then the forgetful functor $F: \C^{(G, \Gamma)} \to \C$, $F(X, r) = X$,
gives rise to a perfect exact sequence of tensor categories
\begin{equation}\label{sec} \Rep G \toto \C^{(G, \Gamma)} \overset{F}\toto \C,
\end{equation} with induced Hopf algebra $H \cong k^G$.
\end{theorem}

\begin{proof} By construction, $\C^{(G, \Gamma)} = \C^T$, where $T =
\bigoplus_{g \in G} \rho^g$ is the Hopf monad  associated to the $(G,
\Gamma)$-crossed tensor category structure on $\C$.
Moreover, the functor $F: \C^{(G, \Gamma)} \to \C$ coincides with the forgetful
functor $\C^T \to \C$.
Since $T$ is a faithful exact endofunctor of $\C$, then the functor $F: \C^{(G,
\Gamma)} \to \C$ is a dominant tensor functor \cite[Lemma 2.1 and Proposition
2.2]{tensor-exact}.
The left exactness of $T$ implies that the functor $F$ has an exact left
adjoint; then $F$ is a perfect tensor functor.

Furthermore, the isomorphisms $\gamma^g_0$, $g \in G$, induce an isomorphism
$T(\uno) =  \bigoplus_{g \in G} \rho^g(\uno) \cong \uno^G$.  Hence $T(\uno)$ is
a trivial object of $\C$, and therefore $T$ is a normal Hopf monad on $\C$. In
view of \cite[Theorem 4.8]{tensor-exact}, the functor $F$ induces an exact
sequence of tensor categories $$\corep H \toto \C^{(G, \Gamma)} \overset{F}\toto
\C,$$ where $H$ is the induced Hopf algebra of $T$, that is, $H$ is the induced
Hopf algebra of the restriction of $T$ to the trivial subcategory of $\C$
\cite[Remark 5.5]{tensor-exact}.
As in the proof of \cite[Theorem 5.21]{tensor-exact}, the restriction of $T$ to
the trivial subcategory $\langle \uno \rangle$ of $\C$ is isomorphic to the Hopf
monad of the trivial action of $G$ on $\langle \uno \rangle$ and therefore $H
\cong k^G$. 
Thus we obtain the perfect exact  sequence \eqref{sec}.   This finishes the
proof of the theorem.
\end{proof}

We shall denote $\Supp \C \subseteq \Gamma$ the support of $\C$,  that is,
$$\Supp \C = \{ s \in \Gamma\vert \, \C_s \neq 0\}.$$ Since the functor
$\otimes$
is faithful in each variable, $\Supp \C$ is a subgroup of $\Gamma$. Moreover,
relation \eqref{rho-partial} implies that $\Supp \C$ is stable under the action
$\lhd$ of $G$ on $\Gamma$.

\begin{proposition}\label{finite-fusion} Let $\C$ be a $(G, \Gamma)$-crossed
fusion category and let
$\D = \C_e$ be the neutral component of $\C$ with respect to the associated
$\Gamma$-grading. Then we have:
\begin{enumerate}\item [(i)] The category $\C^{(G,
\Gamma)}$ is
a finite tensor category if and only if $\D$ is a finite tensor category.
\item[(ii)] The category $\C^{(G, \Gamma)}$ is
a fusion category if and only if $\D$ is a
fusion category and the characteristic of $k$ does not divide the order of $G$.
If this is the case, then we have $$\FPdim \C^{(G, \Gamma)} =
|G| |\Supp \C| \FPdim \D.$$
\end{enumerate}
\end{proposition}

\begin{proof} (i) Assume that $\C^{(G, \Gamma)}$ is a finite tensor category.
Since the forgetful functor $\C^{(G, \Gamma)} \to \C$ is a dominant tensor
functor, then $\C$ is a finite tensor category and therefore so is $\D$. 

Assume, on the other direction, that $\D$ is a finite tensor category. By
Corollary \ref{d-finite-fusion}, $\C$ is a finite tensor category.   By
construction $\C^{(G, \Gamma)} \cong \C^T$, where $T$ is a faithful exact
$k$-linear Hopf monad on $\C$. Then it follows from \cite[Lemma
3.5]{modcat-monads} that $\C^{(G, \Gamma)}$ is a finite tensor category as
well.

\medbreak (ii) Assume that $\D$ is a fusion category and the characteristic of
$k$ does not divide the order of $G$. By Corollary \ref{d-finite-fusion}, $\C$
is also a fusion category. In addition $k^G$ is a cosemisimple Hopf algebra
and therefore $\Rep G = \corep k^G$ is a fusion
category too.    It follows from \cite[Corollary 4.16]{tensor-exact} that
$\C^{(G,
\Gamma)}$ is a fusion category.  

Conversely, assume that $\C^{(G, \Gamma)}$ is a fusion category. Then, by part
(i), $\C$ is a finite tensor category. Consider the forgetful functor $F:
\C^{(G, \Gamma)} \to \C$. Since $F$ is a dominant tensor functor, then 
it maps projective objects of $\C^{(G, \Gamma)}$ to projective objects of $\C$,
by \cite[Theorem 2.5]{EO}. Since $F(\uno) \cong \uno$, and $\uno$ is a
projective object of $\C^{(G, \Gamma)}$, then $\uno$ is a projective object of 
$\C$ and hence $\C$ is a fusion category. Therefore so is its tensor subcategory
$\D$.

In this case, it follows from \cite[Proposition 4.10]{tensor-exact} that
$$\FPdim \C^{(G,
\Gamma)} = \FPdim(\Rep G) \, \FPdim \C = |G| |\Supp \C| \FPdim \D,$$ the last
equality because $\C$ is faithfully graded by $\Supp \C$ with neutral component
$\D$; see \cite[Proposition 8.20]{ENO}. \end{proof}

\begin{proposition}\label{equiv-triv} Let $\C$ be a $(G, \Gamma)$-crossed tensor
category. Then the following statements are equivalent:
\begin{itemize}\item[(i)] The exact sequence \eqref{sec} is an
equivariantization exact sequence.
\item[(ii)] The action $\rhd : \Supp \C \times G \toto G$ is trivial.
\item[(iii)] The action $\lhd : \Supp \C \times G \toto \Supp \C$ is by group
automorphisms.
\end{itemize}
\end{proposition}

In Subsection \ref{g-crossed} we shall further discuss $(G, \Gamma)$-crossed
tensor categories satisfying the equivalent conditions in this proposition.

\begin{proof} Since $\Supp \C$ is a $G$-stable subgroup of $\Gamma$, then
$(\Supp \C, G)$ is a matched pair by restriction. 
As pointed out in Subsection \ref{mpair}, the action $\rhd$ is trivial if and
only if $\lhd$ is an action by group automorphisms. Then (ii) and (iii) are
equivalent.

\medbreak Suppose that the action $\rhd : \Gamma \times G \toto G$ is trivial.
Conditions (a) and (b) imply that, for all $g \in G$,
$\rho^g$ is a tensor functor with tensor structure determined by $\gamma_0^g$
and the isomorphisms $\gamma^g$ in \eqref{gamma}. Moreover, condition
(c) becomes in this case
\begin{equation}((\rho_2^{g, h})_X \otimes (\rho_2^{g, h})_Y) \,
\gamma^g_{\rho^{h}(X), \rho^h(Y)} \, \rho^g(\gamma^h_{X, Y}) =
\gamma^{hg}_{X, Y} \, (\rho_2^{g, h})_{X\otimes Y},
\end{equation}  for all $g, h \in G$ and for all $Y \in \C$. Combining this with
condition (d), we obtain that  $\rho_2^{g, h}: \rho^g \rho^h \to
\rho^{hg}$ are isomorphisms of tensor functors. Therefore $\rho$ is an action by
tensor autoequivalences. Furthermore, the definition of tensor product in
Theorem \ref{tens-prod} reduces in this case to the usual tensor product
\eqref{tens-equiv} in the equivariantization $\C^G$. Hence (ii) implies (i).

\medbreak Suppose that the exact sequence $\Rep G \toto \C^{(G, \Gamma)}
\toto \C$ is an equivariantization exact sequence. Then, by  \cite[Theorem
5.21]{tensor-exact},
the normal Hopf monad $T = \bigoplus_{g \in G}\rho^g$ is cocommutative, that is,
for every morphism $f: T(\uno) \to \uno$ and for every object $X \in \C$, we
have
\begin{equation}\label{T-coconm}(\id_{T(X)} \otimes f) \, T_2(X, \uno) = (f
\otimes \id_{T(X)}) \, T_2(\uno, X): T(X) \to T(X).\end{equation}
Let $s \in \Gamma$, $g \in G$. Restricting both morphisms of \eqref{T-coconm} to
$\rho^g(X) \subseteq T(X)$, $X \in
\C_s$, we get the commutativity of the following diagram:
\begin{equation*}\xymatrix{
\rho^g(X) \ar[d]_{\gamma^g_{\uno, X}}\ar[r]^{\gamma^g_{X, \uno}\qquad } &
\rho^g(X)
\otimes \rho^g(\uno) \ar[d]^{\id \otimes f\vert_{\rho^g(\uno)}}\\
\rho^{s \rhd g}(\uno) \otimes \rho^g(X) \ar[r]_{\qquad \qquad
f\vert_{\rho^{s\rhd
g}(\uno)} \otimes \id \qquad } & \rho^g(X), }
\end{equation*}
for all $g \in G$ and for all morphisms $f: T(\uno) \to \uno$. 

We may apply this to the morphism $f = \gamma^g_0\pi_g$, where $\pi_g$ is the
canonical projection $\pi_g: T(\uno) = \bigoplus_{h \in G}\rho^h(\uno) \to
\rho^g(\uno)$. If $s\rhd g \neq g$, then $f\vert_{\rho^{s\rhd g}(\uno)} =
0$. 

On the other hand, $(\id \otimes f) \, \gamma^g_{X, \uno} = \id_{\rho^g(X)}:
\rho^g(X) \to \rho^g(X)$, by condition (b).

Hence, if $s \in \Supp\C$, we may choose $0 \neq X \in \C_s$, and thus we obtain
$s \rhd g = g$. This shows that (i) implies (ii) and finishes the proof of the
proposition.
\end{proof}

Observe that if $(G, \Gamma)$ is any matched pair, where $\Gamma = \Zz_2$ is the
cyclic group of order 2, then the action $\lhd: \Gamma \times G \to \Gamma$ is
necessarily trivial. As a consequence of Proposition \ref{equiv-triv} we obtain
the following:

\begin{corollary} Let $\C$ be a $(G, \Gamma)$-crossed tensor category, where
$\Gamma \cong \Zz_2$. Then the exact sequence $\Rep G \toto \C^{(G, \Gamma)}
\toto \C$ is an equivariantization exact sequence. \qed
\end{corollary}

Suppose that $\tilde \Gamma$ is a subgroup of $\Gamma$. Then the subcategory
$\C_{\tilde \Gamma} = \bigoplus_{s \in \tilde \Gamma}\C_s$ is a tensor
subcategory of $\C$.

\begin{proposition}\label{g-stable} Let $\tilde \Gamma$ be a subgroup of
$\Gamma$ stable under the action $\lhd$ of $G$. 
Then the actions $\rhd$ and $\lhd$ induce by restriction a matched pair  $(G,
\tilde
\Gamma)$.   The category $\C_{\tilde \Gamma}$ is a $(G, \tilde
\Gamma)$-crossed tensor category by restriction and there is a strict 
embedding of tensor categories $\C_{\tilde \Gamma}^{(G, \tilde\Gamma)} \to
\C^{(G, \Gamma)}$. \end{proposition}

\begin{proof}  Since $\tilde \Gamma$ is stable under the action $\lhd$,
it is clear that $(G, \tilde
\Gamma)$ is a matched pair. Condition \eqref{rho-partial} implies that
$\C_{\tilde \Gamma}$ is stable under the action $\rho$.  It is immediate that
the natural $\tilde\Gamma$-grading and the restriction of $\rho$ make 
 $\C_{\tilde \Gamma}$ into a $(G, \tilde \Gamma)$-crossed tensor category.
Finally, the embedding   $\C_{\tilde \Gamma} \to \C$ induces a strict embedding
of tensor categories
 $\C_{\tilde \Gamma}^{(G, \tilde\Gamma)} \to \C^{(G, \Gamma)}$.
\end{proof}

\begin{remark}\label{neutral} It follows from Definition \ref{crossed-action}
that the neutral
homogeneous component $\C_e$ of $\C$ is a $G$-stable tensor subcategory.
Furthermore, the action of $G$ on $\C$ restricts to an action of $G$ on $\C_e$
by tensor autoequivalences. Therefore $F^{-1}(\C_e) \subseteq \C^{(G, \Gamma)}$
is a tensor subcategory containing $\Rep G$, and in fact $F^{-1}(\C_e) \cong
\C_e^G$ is an equivariantization tensor category. \end{remark}

More generally, let $\overline\Gamma \subseteq \Gamma$ be the subgroup
defined as
$$\overline\Gamma = \{ s \in \Gamma\vert \, s \rhd g = g, \, \forall g \in G
\}.$$
Let $\C_{\overline\Gamma}$ be the tensor subcategory of $\C$ corresponding to
the subgroup
$\overline\Gamma$, that is, $\C_{\overline\Gamma} = \bigoplus_{s \in
\overline\Gamma} \C_s$.

It follows from the relations \eqref{matched} that $\overline\Gamma$ is a
$G$-stable subgroup of $\Gamma$. Let  $\C_{\overline\Gamma}^{(G,
\overline\Gamma)}$
be the tensor subcategory in Proposition \ref{g-stable}. Since $\overline\Gamma$
acts trivially on $G$,  Proposition \ref{equiv-triv} gives us:

\begin{corollary}\label{equiv-sub}  Then the induced exact sequence $\Rep G
\toto \C_{\overline\Gamma}^{(G,
\overline\Gamma)} \toto \C_{\overline\Gamma}$ is an equivariantization exact
sequence. \qed
\end{corollary}

The following proposition and its corollary are dual to Proposition
\ref{equiv-triv} and Corollary \ref{equiv-sub}.

\begin{proposition}\label{gr-triv} Let $\C$ be a $(G, \Gamma)$-crossed tensor
category. Then the following statements are equivalent:
\begin{itemize}\item[(i)] The category $\C^{(G, \Gamma)}$ admits a
$\Gamma$-grading such that the forgetful functor $F: \C^{(G, \Gamma)} \to \C$ is
a $\Gamma$-graded tensor functor.
\item[(ii)] The action $\lhd : \Supp \C \times G \toto \Supp \C$ is trivial.
\item[(iii)] The action $\rhd : \Supp \C \times G \toto G$ is by group
automorphisms.
\end{itemize}
\end{proposition}
 
\begin{proof} The equivalence of (ii) and (iii) follows from relations
\eqref{matched}. We shall show that (i) is equivalent to (ii). 

Assume (ii). For every $s \in \Gamma$, let $\C^{(G, \Gamma)}_s$ denote the full
subcategory of $\C^{(G, \Gamma)}$ of all  objects $(X, r) \in \C^{(G,
\Gamma)}$ such that $X \in \C_s$. If  $(X, r) \in \C^{(G, \Gamma)}_s$ and $(X',
r') \in \C^{(G, \Gamma)}_t$, $s, t \in \Gamma$, then $(X, r) \otimes (X', r') =
(X \otimes X', r'')$ is an object of $\C^{(G, \Gamma)}_{st}$, because $\C_s
\otimes \C_t \subseteq \C_{st}$.  In addition, if $s \neq t$, then $\Hom_\C(X,
X') = 0$ and therefore we obtain $\Hom_{\C^{(G, \Gamma)}}((X, r), (X', r')) =
0$.

Let now $(X, r)$ be any object of $\C^{(G, \Gamma)}$. 
Then, for all $g \in G$, $r^g:\rho^g(X) \to X$ is an isomorphism in $\C$. We
have a decomposition $X \cong \bigoplus_{s \in \Gamma}X_s$, where $X_s \in
\C_s$, for
all $s \in \Gamma$. In view of condition \eqref{rho-partial}, $r^g$ induces by
restriction an isomorphism
$r^g_s:\rho^g(X_s) \to X_{s}$,  for all $g \in G$, $s\in \Gamma$, because the
action $\lhd$ of $G$ on $\Supp \C$ is trivial by assumption. Moreover,  $(X_s,
r_s)$ is an object of $\C^{(G, \Gamma)}$, where $r_s = \{ r^g_s\}_{g \in G}$ is
the restriction of $r$ to $X_s$, and thus $(X, r) \cong \bigoplus_{s \in \Gamma}
(X_s, r_s)$ is a decomposition of $(X, r)$ into a direct sum of objects $(X_s,
r_s) \in \C^{(G, \Gamma)}_s$.  This shows that $\C^{(G, \Gamma)} = \bigoplus_{s
\in \Gamma}\C^{(G, \Gamma)}_s$ is a $\Gamma$-grading in $\C^{(G, \Gamma)}$.
Moreover, for all $(X, r) \in \C^{(G, \Gamma)}_s$, $s \in \Gamma$, we have $F(X,
r) = X \in \C_s$, that is, the functor $F$ is a $\Gamma$-graded tensor functor.
Then we get (i).

\medbreak Conversely, assume that (i) holds. Let $s \in \Supp \C$ and let $0\neq
Y \in \C_s$. Since $F$ is a dominant $\Gamma$-graded tensor functor, there
exists $(X, r) \in \C^{(G, \Gamma)}_s$ such that $Y \subseteq F(X, r) = X$ and
$X \in \C_s$ (see Lemma \ref{f-gr}). In particular, $X \neq 0$ and for all $g
\in G$, $r^g:\rho^g(X) \to X$ is an isomorphism in $\C$. It follows from
condition \eqref{rho-partial} that $s \lhd g = s$, for all $g \in G$. Since $s
\in \Supp \C$  was arbitrary, we get (ii). This shows that  (i) and (ii) are
equivalent and finishes the proof of the proposition. \end{proof}

\begin{remark} The proof of (i) $\Rightarrow$ (ii) in Proposition \ref{gr-triv}
shows that in fact, if $(X, r)$ is an object of $\C^{(G, \Gamma)}$ such that $X$
is a nonzero homogeneous object of $\C$, then the homogeneous degree of $X$ is a
fixed point of $\Gamma$ under the action of $G$. 
\end{remark}

\begin{remark}\label{e-comp} Suppose that the action $\lhd : \Supp \C \times G
\toto \Supp \C$ is trivial. Consider the $\Gamma$-grading of $\C^{G, \Gamma}$
given by Proposition \ref{gr-triv}. Observe that the neutral component $\C^{G,
\Gamma}_e$ of this grading is the category $F^{-1}(\C_e)$. Therefore $\C^{G,
\Gamma}_e \cong \D^G$ is an equivariantization tensor category with respect to
the restriction of the action $\rho$ to the tensor subcategory $\D =\C_e$. See
Remark \ref{neutral}.

Consider the trivial $\Gamma$-grading on $\Rep G$. Let us also quote that, in
this context, the induced exact sequence 
$\Rep G \toto \C^{(G,\Gamma)} \toto \C$
is a \emph{$\Gamma$-graded exact sequence}, that is, both tensor functors
involved are $\Gamma$-graded tensor functors.
\end{remark}

Let $\underline\Gamma \subseteq \Gamma$ be the set of fixed points of $\Gamma$
under the action of $G$. Then $\underline\Gamma$ is a $G$-stable subgroup of
$\Gamma$. Let  $\C_{\underline \Gamma}^{G, \underline\Gamma}$ be the  tensor
subcategory of $\C^{G, \Gamma}$ given by Proposition \ref{g-stable}. Since $G$
acts trivially on the subgroup $\underline\Gamma$, Proposition \ref{gr-triv}
implies the following (c.f. Remark \ref{e-comp}):

\begin{corollary}\label{gr-sub}  The tensor subcategory $\C_{\underline
\Gamma}^{G, \underline\Gamma}$ is a $\underline\Gamma$-graded tensor category
with neutral component $\D^G$, and with respect to the trivial
$\underline\Gamma$-grading on $\Rep G$, the induced exact sequence $\Rep G \toto
\C_{\underline\Gamma}^{(G,
\underline\Gamma)} \toto \C_{\underline\Gamma}$
is a $\underline\Gamma$-graded exact sequence. \qed
\end{corollary}

\begin{remark} Suppose that the neutral component $\D =\C_e$ of $\C$ is a fusion
category. Then $\C^{(G, \Gamma)}$ is also a fusion category, by Proposition
\ref{finite-fusion}. Corollaries \ref{equiv-sub} and \ref{gr-sub} imply that the
fusion subcategories $\C_{\overline\Gamma}^{(G, \overline\Gamma)}$ and
$\C_{\underline\Gamma}^{(G, \underline\Gamma)}$ are, respectively, a
$G$-equivariantization of a group extension of $\D$ and a group extension of a
$G$-equivariantization  of $\D$. In particular, it follows from
\cite[Proposition 4.1]{ENO2} that if $\D$ is weakly group-theoretical, then so
are the fusion subcategories $\C_{\overline\Gamma}^{(G, \overline\Gamma)}$ and
$\C_{\underline\Gamma}^{(G, \underline\Gamma)}$. \end{remark}

\section{$(G, \Gamma)$-crossed braidings}\label{braidings}

Let $(G, \Gamma)$ be a matched pair of finite groups and let $\C$ be $(G,
\Gamma)$-crossed tensor category. We keep the notation in Section
\ref{main-const}.

\begin{definition}\label{crossed-braiding} A \emph{$(G, \Gamma)$-crossed
braiding} in $\C$ is a triple $(c, \vphi, \psi)$, where 

\begin{itemize}
 \item $\vphi, \psi:\Gamma \to G$ are group homomorphisms, satisfying the
following conditions, for all $s, t \in \Gamma$, $g \in G$:
\begin{align} \label{i}& (t^{-1} \lhd \vphi(s^{-1})) st = s \lhd \psi(t),\\
\label{ii}& (t \rhd \psi(s))^{-1} = \psi(s^{-1}\! \lhd \vphi(t^{-1})),\\
\label{iii}& (t^{-1} \rhd \vphi(s^{-1})))^{-1} = \vphi(s\lhd \psi(t)), \\
\label{iv}& \psi(t) g = (t \rhd g) \psi(t \lhd g), \\
\label{v}& g \vphi(s \lhd g)^{-1} = \vphi(s^{-1}) (s \rhd g),
\end{align}

\item $c$ is a collection of natural isomorphisms 
\begin{equation}\label{crossed-bdg}c_{X, Y}: X \otimes Y \to \rho^{t^{-1} \rhd
\vphi(s^{-1})}(Y) \otimes \rho^{\psi(t)}(X), \qquad X \in \C_s, \, Y \in \C_t,
\end{equation}
\end{itemize}

\medbreak For every $s, t \in \Gamma$, let $s \lt t$ and $t \rt s$ be the
elements of $\Gamma$ defined, respectively, by
$$s \lt t = t^{-1} \rhd \vphi(s^{-1}), \qquad t \rt s = s^{-1} \lhd
\vphi(t^{-1}).$$

The isomorphisms $c_{X, Y}$ are subject to the commutativity of the
following diagrams  (when there is no
ambiguity, we omit subscripts to denote morphisms):

\medbreak
\noindent (1) For all $g \in G$, $s, t \in \Gamma$, $X \in \C_s$, $Y \in \C_t$, 


\begin{equation*}\xymatrix{
\rho^g(X\otimes Y)  \ar[d]_{\rho^g(c)} \ar[r]^{\gamma^g \qquad \qquad} &
\rho^{t\rhd g}(X)  \otimes \rho^g(Y) \ar[d]^{c} \\
\rho^g\!(\rho^{s\lt t}\!(Y) \otimes
\rho^{\psi(t)}\!(X))\ar[d]_{\gamma^g} & \rho^{(s\lhd
(t\rhd g))\lt(t\lhd g)}\!\rho^g\!(\!Y\!) \otimes \rho^{\psi(t\lhd
g)}\!\rho^{t\rhd g}\!(\!X\!)\ar[d]^{\rho_2
\otimes \rho_2} \\
\rho^{(s \lhd \psi(t))\rhd g}\rho^{s\lt t}(Y) \otimes
\rho^g\rho^{\psi(t)}(X) \ar[r]_{\;\rho_2 \otimes \rho_2}  & \rho^{(s\lt t)
((s\lhd \psi(t))\rhd g)}(Y) \otimes \rho^{\psi(t)g}(X)
}\end{equation*}

\medbreak
\noindent (2) For all $s, t, u \in \Gamma$, $X \in \C_s$, $Y \in \C_t$, $Z \in
\C_u$,


\begin{equation*}\xymatrix{
X\otimes Y \otimes Z \ar[dd]_{\id \otimes c} \ar[r]^{c_{X \otimes Y, Z} \qquad
\qquad} & \rho^{st \lt u}(Z)  \otimes \rho^{\psi(u)}(X \otimes
Y)
\ar[d]^{\id\otimes \gamma^{\psi(u)}} \\
& \rho^{st \lt u}(Z) \otimes \rho^{t\rhd\psi(u)}(X) \otimes
\rho^{\psi(u)}(Y) \\
X \otimes \rho^{t \lt u}\!(Z) \otimes \rho^{\psi(u)}\!(Y)
\ar[r]_{c
\otimes \id \qquad \qquad\quad }  & 
\rho^{s \lt (t\rt u)^{-1}}\!\rho^{t \lt u}\!(Z) \otimes
\rho^{\psi(t \rt u)^{-1}}\!(X) \otimes \rho^{\psi(u)}\!(Y)
\ar[u]_{\rho_2\otimes \id \otimes \id} }\end{equation*}

\medbreak \noindent (3) For all $s, t, u \in \Gamma$, $X \in \C_s$, $Y \in
\C_t$, $Z \in \C_u$,


\begin{equation*}\xymatrix{
X\otimes Y \otimes Z \ar[dd]_{c \otimes \id} \ar[r]^{c_{X,
Y\otimes Z} \qquad \qquad} & \rho^{s \lt tu}(Y \otimes Z)
\otimes
\rho^{\psi(tu)}(X) \ar[d]^{\gamma^{(tu)^{-1} \rhd s} \otimes \id} \\
& \rho^{s \lt t}(Y) \otimes
\rho^{s \lt tu}(Z)
\otimes \rho^{\psi(tu)}(X) \\
\rho^{s \lt t}(Y) \otimes \rho^{\psi(t)}(X) \otimes Z
\ar[r]_{\id
\otimes c\qquad \qquad }  & 
\rho^{s \lt t}(Y) \otimes \rho^{(s \lhd
\psi(t)) \lt u}(Z) \otimes \rho^{\psi(u)}\rho^{\psi(t)}(X)
\ar[u]_{\id \otimes \id \otimes \rho_2} }\end{equation*}

\end{definition}

\begin{remark} Let $\C$ be a $(G, \Gamma)$-crossed tensor category
and let $\vphi$ and $\psi: \Gamma \to G$ be  maps. Assume in addition that the
$\Gamma$-grading on $\C$ is faithful and the $G$-action is faithful. Conditions
(1), (2) and (3) in Definition
\ref{crossed-braiding} on the natural isomorphism $c$  imply that the maps
$\vphi$ and $\psi$ are group homomorphisms and that they satisfy the relations
\eqref{i}--\eqref{v}.
This can be shown  with an argument similar
to that in Remark \ref{suff-mp}.

For instance, the relations \eqref{matched} imply that $(t^{-1}\lhd g)^{-1} =
(t\lhd (t^{-1}\rhd g))$, for all $t \in \Gamma$, $g \in G$. The existence of an
isomorphism like in \eqref{crossed-bdg} makes it necessary that condition
\eqref{i} in Definition \ref{crossed-braiding} holds, when the $\Gamma$-grading
on $\C$ is faithful.
\end{remark}

\begin{remark} Let $\C$ be $(G, \Gamma)$-crossed tensor
category and suppose  $(c, \vphi, \psi)$ is a $(G, \Gamma)$-braiding in $\C$.  
Conditions (2) and (3) in Definition \ref{crossed-braiding} imply that the
neutral homogeneous component $\D = \C_e$ of $\C$ is a braided tensor category
with braiding
induced by the restriction of the natural isomorphism $c$.
 \end{remark}
 
\subsection{Crossed braidings and the set-theoretical QYBE}\label{set-YBE}

Let $(G, \Gamma)$ be a matched pair of groups and let $G \Join \Gamma$ be the
associated group   (see Subsection \ref{mpair}). We shall identify $G$ and
$\Gamma$ with
subgroups of $G \Join \Gamma = G \times \Gamma$ in the natural way.
Thus $G \Join \Gamma$ is endowed with an exact factorization into its subgroups
$G$ and $\Gamma$.

\medbreak 
The exact factorization in $G \Join \Gamma$ induces actions of $G$ and $\Gamma$
on
each other, denoted $^sg$, $s^g$, $^gs$, $g^s$, $s \in \Gamma$, $g \in G$, which
are uniquely determined by the relations 
\begin{equation}\label{ac-lyz}
s g = {}^s\!g  \, s^g, \qquad g s = {}^g\!s \, g^s,
\end{equation}
in $G \Join \Gamma$. See \cite[Section 2]{lyz}.

From the definition of the group $G \Join \Gamma$, we obtain the following
relations:
\begin{equation*}  ^sg = s \rhd g, \qquad s^g = s \lhd g,\qquad 
 ^gs = (s^{-1}\lhd g^{-1})^{-1}, \qquad g^s =  (s^{-1} \rhd g^{-1})^{-1},
\end{equation*}
for all $g \in     G$, $s\in \Gamma$.

\medbreak Let $\vphi, \psi: \Gamma \to G$ be group
homomorphisms.   
The conditions \eqref{i}--\eqref{v} in Definition \ref{crossed-braiding} are
equivalent, respectively, to the following conditions:
\begin{align} \label{8}& st = {}^{\psi(s)}\!t \, s^{\vphi(t)},\\
\label{6}& \psi(s)^{t} = \psi(s^{\vphi(t)}),\\
\label{7}& {}^{s}\!\vphi(t) = \vphi({}^{\psi(s)}\!t), \\
\label{9}& \psi({}^gt) g^t = \psi(t) g, \\
\label{10}& \vphi({}^gs)g^s = g \vphi(s),
\end{align}
for all $s, t \in \Gamma$, $g \in G$.  Compare with \cite[Proposition 1]{lyz}.  
 
An alternative formulation for the conditions on the data $(G, \Gamma, \vphi,
\psi)$, in terms of group actions by
automorphisms and 1-cocycles, is explained in \cite[Theorem 2]{lyz}.

\begin{remark}\label{b-inv} 
Consider the map $b_{\vphi, \psi}: \Gamma \times \Gamma
\to \Gamma \times \Gamma$, given by
$$b_{\vphi, \psi} (s, t) = ((t^{-1} \lhd \vphi(s^{-1}))^{-1}, s \lhd \psi(t)),
\quad
s, t \in \Gamma.$$ 
In terms of the actions \eqref{ac-lyz}, this map has the following expression:
\begin{equation*}b_{\vphi, \psi}  (s, t) = ({}^{\vphi(s)}\!t, s^{\psi(t)}),
\qquad s, t \in \Gamma.
\end{equation*}
It turns out that  $b_{\vphi, \psi} (s, t)$ coincides with the map $\mathcal
R^{-1}(t, s)$, where $\mathcal R: \Gamma \times \Gamma \to
\Gamma \times \Gamma$ is the (invertible) set-theoretical solution of the
 QYBE on the set $\Gamma$ given in \cite{lyz2},
corresponding to the actions ${}^{\vphi(s)}\!t$ and $s^{\psi(t)}$ of $\Gamma$ on
itself. The relevant condition for the result of \cite{lyz2} is  \eqref{8} or,
equivalently,  \eqref{i}.  In particular  the map $b_{\vphi, \psi}$ is
bijective.
\end{remark}

\subsection{Braiding in the category $\C^{(G, \Gamma)}$}

We next show that a $(G, \Gamma)$-crossed braiding in $\C$ induces a
braiding in the associated tensor category $\C^{(G, \Gamma)}$. 

\begin{theorem}\label{crossed-braided} Let $\C$ be $(G, \Gamma)$-crossed tensor
category and let $(c, \vphi, \psi)$ be a $(G, \Gamma)$-braiding in $\C$. Then
$\C^{(G, \Gamma)}$ is a braided tensor category with brading $$c_{(X, r), (Y,
l)}: (X, r) \otimes (Y, l) \to (Y, l) \otimes (X, r),$$ defined componentwise by
the isomorphisms 
\begin{equation}\label{formula-sigma}(l^{t^{-1}\rhd \vphi(s^{-1})} \otimes
r^{\psi(t)}) \, c_{X_s, Y_t}: X_s \otimes Y_t \to Y_{t\lhd (t^{-1}\rhd
\vphi(s^{-1}))} \otimes X_{s\lhd \psi(t)},\end{equation}
where $X = \bigoplus_{s \in \Gamma}X_s$ and $Y =\bigoplus_{t \in \Gamma}Y_t$.
\end{theorem}

\begin{proof} Recall that $\C^{(G, \Gamma)} = \C^T$, where $T = \bigoplus_{g \in
G}\rho^g$ is the Hopf monad in Theorem \ref{tens-prod}. The natural isomorphisms
$c_{X, Y}: X \otimes Y \to \rho^{t^{-1} \rhd \vphi(s^{-1})}(Y) \otimes
\rho^{\psi(t)}(X)$, $X \in \C_s$, $Y \in \C_t$, induce canonically a natural
transformation 
$$R_{X, Y}: X \otimes Y \to \bigoplus_{g, h \in G}\rho^g(Y)\otimes \rho^h(X).$$
The commutativity of the diagrams (1), (2) and (3) in Definition
\ref{crossed-braided} imply, respectively, that the natural transformation $R$
satisfies conditions \eqref{r1}, \eqref{r2} and \eqref{r3}.

Let $(X, r), (Y, l) \in \C^{(G, \Gamma)}$. Then the natural transformation
$$R^{\#}_{(X, r), (Y, l)} = (l \otimes r) R_{X, Y}: X \otimes Y \to Y \otimes
X,$$
is given componentwise by isomorphisms 
$$(l^{t^{-1}\rhd \vphi(s^{-1})} \otimes r^{\psi(t)}) \, c_{X_s, Y_t}: X_s
\otimes Y_t
\to Y_{t\lhd (t^{-1}\rhd \vphi(s^{-1}))} \otimes X_{s\lhd \psi(t)},$$ 
where $X = \bigoplus_{s \in \Gamma}X_s$ and $Y =\bigoplus_{t \in \Gamma}Y_t$.
Recall that $t\lhd (t^{-1}\rhd \vphi(s^{-1})) = (t^{-1}\lhd
\vphi(s^{-1}))^{-1}$, for all
$s, t \in \Gamma$. It was observed in Remark \ref{b-inv} that the map $b_{\vphi,
\psi}: \Gamma \times \Gamma \to \Gamma \times \Gamma$, defined by 
$b_{\vphi, \psi} (s, t) = ((t^{-1} \lhd \vphi(s^{-1}))^{-1}, s \lhd\psi(t))$, 
is bijective. 
This implies that $R^\#$ is an
isomorphism. 

We have thus shown that $T$ is a quasitriangular Hopf monad on $\C$. Therefore
$\C^{(G, \Gamma)}$ is a braided tensor category with the braiding induced by the
$R$-matrix $R$, which is easily seen to coincide with \eqref{formula-sigma}.
This finishes the proof of the theorem. \end{proof}

\section{Some families of examples}\label{ejemplos}

\subsection{$G$-crossed categories}\label{g-crossed}

Let $G$ be a finite group. Then there is a matched pair $(G, \Gamma)$, where
$\Gamma =G$, $\lhd: G \times \Gamma \to G$ is the trivial action and $\rhd: G
\times \Gamma \to \Gamma$ is the adjoint action.

A $(G, G)$-crossed tensor category is the same as a $G$-graded
tensor category $\C =\bigoplus_{g \in G}\C_g$, endowed with a $G$-action by
tensor autoequivalences $\rho:\underline G^{\op} \to
\underline{\Aut}_{\otimes}(\C)$ such that $\rho^g(\C_h) = \C_{g^{-1}hg}$, for
all $g, h \in G$.

Thus, as a monoidal category, a $(G, G)$-crossed tensor category is a
$G$-crossed  category as defined in \cite[Section 3.2]{TV}. See also
\cite[Chapter VI]{turaev}.

\medbreak In this case the exact sequence of tensor categories given by Theorem
\ref{exact-sequence},
$$\Rep G \to \C^{(G, G)} \to \C,$$  is an
equivariantization exact sequence; see Proposition
\ref{equiv-triv}. 

\begin{remark} Consider a matched pair $(G, \Gamma)$ such that the action $\rhd$
is
trivial or, equivalently, such that the action $\lhd$ is by group automorphisms.
In this context, the notion of $(G, \Gamma)$-crossed fusion category is not new.
In fact, any $(G, \Gamma)$-crossed fusion category associated to such a matched
pair can be recovered from the $G$-crossed categories of \cite{TV}.

This is due to the well-known fact that any action by group automorphisms can be
recovered from an adjoint action, and can be formulated as follows.

\medbreak Suppose that the action $\lhd: \Gamma \times G \to \Gamma$ is by group
automorphisms. Let $S = \Gamma \rtimes G$ be the semidirect product associated
to this action, so that the following relations hold in $S$:
\begin{equation}\label{adj}
g^{-1}sg = s \lhd g, 
\end{equation}
 for all $s \in \Gamma$, $g \in G$.

Consider an $S$-crossed (tensor) category $\C = \bigoplus_{s \in S}\C_s$. Since
$\Gamma$ is a normal subgroup of $S$ then the tensor subcategory $\C_\Gamma =
\bigoplus_{s \in \Gamma}\C_s$ is stable under the adjoint action of $S$.  Hence
it is also stable under the adjoint action of $G$.    Relation \eqref{adj},
together with the conditions defining a $G$-crossed category in \cite[Section
3.2]{TV}, imply that 
$\C_\Gamma$ is a $(G, \Gamma)$-crossed tensor category.

\medbreak    Conversely, suppose that $\C = \bigoplus_{s \in \Gamma}\C_s$ is a
$(G, \Gamma)$-crossed tensor category. Condition \eqref{rho-partial} in
Definition \ref{crossed-action}  implies that 
$$\rho^g(\C_s) = \C_{g^{-1}sg},$$ for all $s \in \Gamma$, $g \in G$, in view of
\eqref{adj}.

The $\Gamma$-grading on $\C$ induces an $S$-grading $\C = \bigoplus_{s \in
S}\C_s$, letting $\C_s : = 0$, for all $s \in S \backslash \Gamma$.

Similarly, the action  $\rho:
\underline{G}^{\op} \to \underline{\Aut}(\C)$ which, by Proposition
\ref{equiv-triv}, is in this case an action by tensor autoequivalences, 
induces an action by tensor autoequivalences $\tilde\rho: \underline{S}^{\op}
\to \underline{\Aut}(\C)$ in the form ${\tilde\rho}^s = \rho^{\overline s}$, 
for all $s \in S$, where $\overline s \in G$ denotes the image of $s$ under the
canonical projection $S \to G$.

The remaining conditions in Definition \ref{crossed-action} imply that $\C$
becomes in this way an $S$-crossed category.
\end{remark}

Recall that a $G$-braiding in a $G$-crossed category $\C$ is a collection of
natural isomorphisms $\alpha_{X, Y}: X \otimes Y \to Y\otimes \rho^t(X)$, $Y \in
\C_t$, called a \emph{$G$-braiding}, satisfying  appropriate compatibility
conditions. See \cite{turaev}, \cite[Subsection 3.3]{TV}.  

\begin{proposition}\label{data-gb} Let $\C$ be a $G$-crossed tensor category.
Then the following data are equivalent:
\begin{enumerate}
\item[(i)] A $G$-braiding $c$ in $\C$. 
\item[(ii)] A $(G, G)$-crossed braiding $(c, \vphi, \psi)$ in $\C$, where 
$\psi = \id_G: G \to G$ and $\vphi: G \to G$ is the trivial group homomorphism.
\end{enumerate}\end{proposition}

Note that the trivial homomorphism $\vphi$ and the identity homomorphism
$\psi$ satisfy conditions \eqref{i}--\eqref{v} in Definition
\ref{crossed-braiding}. The map $b_{\vphi, \psi}$ is given in this case by
$$b_{\vphi, \psi}(s, t) = (t, t^{-1}st), \qquad s, t \in G.$$

\begin{proof} It is enough to observe that the commutativity of the diagrams
(1)--(3) in Definition \ref{crossed-braiding} in the case where $\psi$ is the
identity homomorphism and $\vphi: G \to G$ is the trivial group homomorphism, is
equivalent to the commutativity of the diagrams in \cite[Subsection 3.3]{TV}.   
\end{proof}

\begin{remark} Let $\C$ be a $G$-braided tensor category regarded as a $(G,
G)$-crossed tensor category. Suppose  $(c, \vphi, \psi)$ is any $(G,
G)$-braiding in $\C$ where $\psi = \id_G$. It follows from conditions
\eqref{i}--\eqref{v} that $\vphi$ is a group homomorphism $\vphi: G \to Z(G)$. 
\end{remark}

\subsection{Abelian exact sequences of Hopf algebras}\label{subs-ab}

Consider a matched pair of finite groups $(G, \Gamma)$.
Let also $\sigma: G \times G \to (k^*)^\Gamma$ and $\tau: \Gamma \times \Gamma
\to (k^*)^ G$ be  normalized
2-cocycles, that is, using the notation $\sigma_s(g, h) = \sigma(g, h)(s)$ and
$\tau_g(s, t) = \tau(s, t)(g)$, $s, t\in \Gamma$, $g, h \in  G$, the following
relations hold:
\begin{align} &\sigma_{s\vartriangleleft g}(h, l)\sigma_s(g, hl) = \sigma_s(g,
h)\sigma_s(gh, l),\\
& \sigma_s(e, g) = \sigma_s(g, e) = 1,\\
&\tau_g(st, u) \tau_{u\vartriangleright g}(s, t) = \tau_g(t, u)\tau_g(s, tu),\\
&\tau_g(e, s) = \tau_g(s, e) = 1,
\end{align}
for all $g, h, l \in  G$, $s, t, u \in \Gamma$.

\medbreak Assume in addition that $\sigma$ and $\tau$ satisfy the following
compatibility conditions:
\begin{equation}\sigma_{st}(g, h) \tau_{gh}(s, t) =
\sigma_s(t\!\vartriangleright \!g,
(t \!\vartriangleleft \!g) \!\vartriangleright \!h)  \sigma_t(g, h) \tau_g(s, t)
\tau_h(s\!\vartriangleleft \!(t\vartriangleright \!g), t\!\vartriangleleft \!g),
\end{equation}
\begin{equation} \sigma_e(g, h) = \tau_e(s, t) = 1,\end{equation}
for all $s, t \in \Gamma$, $g, h \in  G$.

\medbreak Then the vector space $H = k^\Gamma \otimes k  G$  becomes a
 Hopf algebra with the crossed product algebra structure and crossed
coproduct coalgebra structure, denoted  $H = k^\Gamma \, {}^{\tau}\#_{\sigma}k
G$. The multiplication and comultiplication of $H$ are defined, for all
$g,h\in \Gamma$, $g, h\in  G$, in the form
\begin{align}\label{mult} (e_s \# g)(e_t \# h) & = \delta_{s \vartriangleleft g,
h}\, \sigma_s(g,
h) e_s \# gh, \\
\label{delta} \Delta(e_s \# g) & = \sum_{tu=s} \tau_g(t, u)\, e_t
\# (u \vartriangleright g) \otimes e_u\# g.
\end{align}
It is well-known that $H$ is a semisimple Hopf algebra if and only if the
characteristic of $k$ does not divide the order of $G$.

\medbreak Let $i = \id \otimes u: k^\Gamma \to H$ and $p = \epsilon \otimes \id:
H \to k G$ be the canonical Hopf algebra maps. Then we
have an exact sequence of Hopf algebras 
\begin{equation}\label{abeliana}
k \toto k^\Gamma \overset{i}\toto H
\overset{p}\toto k G \toto k. 
\end{equation}

By \cite[Proposition 3.9]{tensor-exact} this
exact sequence gives rise to an exact sequence of tensor categories
\begin{equation}\label{abelian}\Rep G \overset{p^*}\toto \Rep H
\overset{i^*}\toto \C(\Gamma),\end{equation} where $\C(\Gamma) = \Rep
k^{\Gamma}$ is the
category of finite dimensional $\Gamma$-graded vector spaces.

\medbreak The category $\C(\Gamma)$ is a $(G, \Gamma)$-crossed fusion category
with respect to the following data:
\begin{enumerate}\item[(a)] The $\Gamma$-grading $\C(\Gamma) = \bigoplus_{s \in
\Gamma} \C(\Gamma)_s$, where, for all $s \in \Gamma$, $\C(\Gamma)_s$ is the
category of finite dimensional vector spaces of degree $s$.
\item[(b)] The action $\rho: \underline{G}^{\op} \to
\underline{\Aut}(\C(\Gamma))$ is given by $\rho^g(V) = V$ with $G$-grading
$\rho^g(V)_s = V_{s \lhd g}$.

The monoidal structure of $\rho$ is given by $\rho_0 = \id: \rho^e \to
\id_{\C(\Gamma)}$, and $\rho_2^{g, h} = \sigma(h, g)^{-1}: \rho^g\rho^h(V) \to
\rho^{hg}(V)$, that is, $$\rho_2^{g, h}(v) = \sigma_{|v|}(h, g)^{-1}v, $$ for
every homogeneous element $v \in V$ of degree $|v|$.

\item[(c)] For all $U \in \C(\Gamma)$, $V \in \C(\Gamma)_s$, the natural
isomorphisms $\gamma^g_{U, V}: \rho^g(U \otimes V) \to \rho^{s\rhd g}(U) \otimes
\rho^g(V)$, are given by
$$\gamma^g_{U,V} (u \otimes v) = \tau_g(|u|, s) \, u \otimes v,$$ on homogeneous
elements $u \in U$ of degree $|u|$.

\item[(d)] The isomorphisms $\gamma^g_0: \rho^g(k) = k \to k$ are identities,
for all $g \in G$.
\end{enumerate}

The next theorem relates the tensor category associated to the $(G,
\Gamma)$-crossed tensor category $\C(\Gamma)$ with the Hopf algebra $H$.

\begin{theorem}   There is a strict equivalence of tensor categories
$$\C(\Gamma)^{(G, \Gamma)} \cong \Rep H.$$
 \end{theorem}

\begin{proof} Since $H = k^{\Gamma}\#_{\sigma}kG$ is a crossed product as an
algebra, it
follows from \cite[Proposition 3.2]{ext-ty} that $\rho$ is an action by
$k$-linear autoequivalences and there is an equivalence of  $k$-linear
categories $K: \Rep H \cong \C(\Gamma)^G = \C(\Gamma)^{(G, \Gamma)}$, where for
all $H$-module $W$,  $K(W) = (W\vert_{k^{\Gamma}}, g^{-1}\vert_W)$. The inverse
equivalence maps an object $(V, r)$ of $\C(\Gamma)^{(G, \Gamma)}$ to the vector
space $V$ endowed with the $H$-action $(e_s \# g) . v = (r^g)^{-1}(v_s)$, $v \in
V$.

It is straightforward to verify that $K$ is a strict equivalence of tensor
categories. This proves the theorem. \end{proof}

\begin{remark} Consider the case where the exact sequence \eqref{abeliana} is a
split exact sequence. This corresponds to the situation where $\sigma$ and
$\tau$ are the trivial 2-cocycles.

Regard the category $\C(\Gamma)$ as a $(G, \Gamma)$-crossed tensor category as
above. Suppose $(c, \vphi, \psi)$ is a $(G, \Gamma)$-braiding in $\C(\Gamma)$. 
It follows from \cite[Theorem 1]{lyz} that the compatibility conditions between
$\vphi$ and $\psi$ given in Definition \ref{crossed-braiding} imply that  pairs
$(\vphi, \psi)$, satisfying the compatibility conditions in Definition \ref{crossed-braiding},
are in bijective correspondence with positive 
 quasitriangular structures in the Hopf algebra $H$. In fact, the conditions in
\cite[Theorem 1]{lyz} are equivalent to the conditions \eqref{i}--\eqref{v}, in
view of \cite[Proposition 1]{lyz}.     See  Subsection \ref{set-YBE}.
\end{remark}

\bibliographystyle{amsalpha}

\begin{thebibliography}{AAAA}

\bibitem{BK} B. Bakalov, A. Kirillov, Jr., \emph{Lectures on tensor categories
and modular functors}, 
University Lecture Series \textbf{21}, Am. Math. Soc.,  Providence, RI, 2001.

\bibitem{bv} A. Brugui\`{e}res, A. Virelizier, \emph{Hopf monads}, Adv. Math.
\textbf{215}, 679--733 (2007).

\bibitem{blv} A. Brugui\`{e}res, S. Lack, A. Virelizier, \emph{Hopf monads on
monoidal categories}, Adv. Math. \textbf{227}, 745--800 (2011).

\bibitem{tensor-exact} A. Brugui\`{e}res, S. Natale, \emph{Exact sequences of
tensor categories},
Int. Math. Res. Not. \textbf{2011}  (24), 5644--5705 (2011).

\bibitem{indp-exact} A. Brugui\`{e}res, S. Natale, \emph{Central exact sequences
of tensor categories, equivariantization and applications},  J. Math. Soc. Japan
\textbf{66}, 257-287 (2014).

\bibitem{egno}  P. Etingof,  S. Gelaki, D. Nikshych, V. Ostrik,
\emph{Tensor categories},  Lecture Notes, MIT 18.769,
2009.

\bibitem{ENO}  P. Etingof,  D. Nikshych, V. Ostrik,
\emph{On fusion categories},  Ann. Math. \textbf{162}, 581--642 (2005).

\bibitem{ENO2}  P. Etingof, D. Nikshych, V. Ostrik,
\emph{Weakly group-theoretical and solvable fusion categories},
Adv. Math \textbf{226},    176--205   (2011).

\bibitem{ENO3}  P. Etingof,  D. Nikshych, V. Ostrik,
\emph{Fusion categories and homotopy theory}, 
Quantum Topol. \textbf{1} 209--273 (2010).

\bibitem{EO}  P. Etingof, V. Ostrik,
\emph{Finite tensor categories},
Mosc. Math. J. \textbf{4}, 627--654 (2004).

\bibitem{ess}  P. Etingof, T.  Schedler, A. Soloviev, \emph{Set-theoretical
solutions to the quantum Yang-Baxter equation},  Duke Math. J. \textbf{100},
169--209 (1999).

\bibitem{kac} G. I. Kac, \emph{Extensions of groups to ring groups}, Math. USSR
Sbornik \textbf{5}, 451--474 (1968).

\bibitem{kirillov}  A. Kirillov, Jr., \emph{Modular categories and orbifold
models II}, preprint
\texttt{arXiv:\-0110221}.

\bibitem{lyz2} J.-H. Lu,  M. Yan, Y. Zhu, \emph{On the set-theoretical
Yang-Baxter equation}, 
Duke Math. J. \textbf{104},  1--18 (2000).

\bibitem{lyz} J.-H. Lu,  M. Yan, Y. Zhu, \emph{Quasi-triangular structures on
Hopf algebras with positive bases}, 
Contemp. Math. \textbf{267}, 339--356 (2000).

\bibitem{majid} S. Majid, \emph{Physics for algebraists: Non-commutative and
non-cocommutative Hopf algebras by a bicrossproduct construction}, J. Algebra
\textbf{130}, 17--64 (1990).

\bibitem{masuoka} A. Masuoka,  \emph{Hopf algebra extensions and cohomology},
Math. Sci. Res. Inst. Publ.
\textbf{43}, 167--209 (2002).

\bibitem{modcat-monads} M.  Mombelli, S. Natale, \emph{Module categories over
equivariantized tensor categories}, preprint (2014).

\bibitem{mueger-crossed} M M\" uger, \emph{Galois extensions of braided tensor
categories and braided
crossed $G$categories},  J. Algebra \textbf{277}, 256--281 (2004).

\bibitem{ext-ty} S. Natale, \emph{Hopf algebra extensions of group algebras and
Tambara-Yamagami categories}, Algebr. Represent. Theory \textbf{13},  673--691
(2010).

\bibitem{takeuchi} M. Takeuchi, \emph{Matched pairs of groups and bismash
products of Hopf algebras},  Commun. Algebra
\textbf{9}, 841--882 (1981).

\bibitem{turaev} V. Turaev, \emph{Homotopy quantum field theory}, EMS Tracts in
Mathematics \textbf{10}, European Mathematical Society, Zurich, 2010.

\bibitem{s} A. Soloviev, \emph{Non-unitary set-theoretical solutions to the
quantum Yang-Baxter equation}, 
Math. Res. Lett. \textbf{7}, 577--596 (2000).

\bibitem{TV} V. Turaev, A. Virelizier, \emph{On the graded center of graded
categories}, J. Pure Appl. Algebra \textbf{217}, 1895--1941 (2013).


\end{thebibliography}

\end{document}